\documentclass[11pt]{amsart}
\usepackage{amssymb}
\usepackage{amscd}
\usepackage{amsmath}
\usepackage{amsmath}
\usepackage[all]{xy}
\usepackage{mathrsfs}

\usepackage{color}

\usepackage{bm}
\theoremstyle{plain}
\newtheorem{theorem}{Theorem}[section]
\newtheorem{lemma}[theorem]{Lemma}
\newtheorem{corollary}[theorem]{Corollary}
\newtheorem{proposition}[theorem]{Proposition}
\newtheorem{conjecture*}{Conjecture}
\newtheorem{theorem*}{Theorem}
\newtheorem{corollary*}{Corollary}

\newtheorem{conjecture}[theorem]{Conjecture}
\theoremstyle{definition}

\newtheorem{hypothesis}[theorem]{Hypothesis}
\theoremstyle{definition}
\newtheorem{definition}[theorem]{Definition}
\newtheorem{remark}[theorem]{Remark}

\newtheorem*{acknowledgments}{Acknowledgements}

\font\russ=wncyr10  1
\def\sha{\hbox{\russ\char88}}

\DeclareMathOperator{\Gal}{Gal}

\DeclareMathOperator{\Hom}{Hom}

\DeclareMathOperator{\Spec}{Spec}
\DeclareMathOperator{\N}{N}

\DeclareMathOperator{\im}{im}

\newcommand{\CC}{\mathbb{C}}

\newcommand{\GG}{\mathbb{G}}

\newcommand{\QQ}{\mathbb{Q}}

\newcommand{\RR}{\mathbb{R}}

\newcommand{\ZZ}{\mathbb{Z}}

\newcommand{\cL}{\mathcal{L}}

\newcommand{\cG}{\mathcal{G}}

\newcommand{\cK}{\mathcal{K}}

\newcommand{\cO}{\mathcal{O}}

\newcommand{\cR}{\mathcal{R}}
\newcommand{\cU}{\mathcal{U}}

\newcommand{\fX}{\mathfrak{X}}
\newcommand{\fz}{\mathfrak{z}}

\newcommand{\per}{2\pi \sqrt{-1}}

\newcommand{\catname}[1]{\textnormal{{\textsf{#1}}}}

\newcommand{\DR}{\catname{R}}
\newcommand{\DL}{\catname{L}}

\newcommand{\rgamma}{\DR\Gamma}

\newcommand{\lotimes}{\otimes^{\DL}}

\makeatletter
 
  \@addtoreset{equation}{subsection}
\makeatother

\begin{document}

\title[]{On functional equations of Euler systems}

\author{David Burns and Takamichi Sano}

\begin{abstract} We establish precise relations between Euler systems that are respectively associated to a $p$-adic representation $T$ and to its Kummer dual $T^*(1)$. Upon appropriate specialization of this general result, we are able to deduce the existence of an Euler system of rank $[K:\QQ]$ over a totally real field $K$ that both interpolates the values of the Dedekind zeta function of $K$ at all positive even integers and also determines all higher Fitting ideals of the Selmer groups of $\mathbb{G}_m$ over abelian extensions of $K$. This construction in turn motivates the formulation of a precise conjectural generalization of the Coleman-Ihara formula and we provide supporting evidence for this conjecture. \end{abstract}

\address{King's College London,
Department of Mathematics,
London WC2R 2LS,
U.K.}
\email{david.burns@kcl.ac.uk}

\address{Osaka City University,
Department of Mathematics,
3-3-138 Sugimoto\\Sumiyoshi-ku\\Osaka\\558-8585,
Japan}
\email{sano@sci.osaka-cu.ac.jp}

\maketitle

\tableofcontents

\section{Introduction}

\subsection{Background and results}

Let $p$ be a prime number and $T$ a $p$-adic representation over a number field $K$. Then there is considerable interest in the construction of `special elements' that lie in the higher exterior powers (or exterior power biduals) of the cohomology groups of $T$ over abelian extensions of $K$ and can be explicitly linked to the values of derivatives of $L$-functions that are related to $T$. In the best case, a family of such elements constitutes an Euler system of the appropriate rank for $T$ and hence controls the structure of Selmer modules associated to $T$ (via the general theory of \cite{bss}). However, despite their great importance, there are still even today very few known constructions of Euler systems.

With a view to better understanding both the basic properties and possible constructions of such families, our main interest in this article is to explore relations that should exist between the special elements associated to $T$ (and suitable auxiliary data) and the special elements associated to its Kummer dual $T^*(1) := \Hom_{\ZZ_p}(T,\ZZ_p(1))$.

Our approach uses the theory of `basic' Euler systems that was introduced in \cite{sbA}. We recall, in particular, that the latter theory involves a construction of  `vertical determinantal systems' that arise from a detailed study of the Galois cohomology complexes of representations.

In this regard, the main theoretical advance that we shall make is to introduce for a general representation $T$ over an arbitrary  number field $K$ a natural `local' analogue of the notion of vertical determinantal systems and to show (in Theorem \ref{funceq}) that this can be used to establish a `functional equation' for the vertical determinantal systems, and hence basic Euler systems, that are associated to $T$ and to $T^*(1)$.

To give a first concrete application of this general result, we specialize to representations of the form $\ZZ_p(j)$ for an odd prime $p$ and suitable integers $j$. In particular, if $K$ is a totally real field in which $p$ does not ramify, then we can construct a canonical local vertical determinantal system in this setting by means of a canonical `higher rank Coleman map' (see Remark \ref{higher col}) and then use the pre-image under this map of the Deligne-Ribet $p$-adic $L$-function to prove the following result (which is later stated precisely as Theorem \ref{main}).

\begin{theorem*}[Theorem \ref{main}]\label{th1}
Let $K$ be a totally real field and $p$ an odd prime that does not ramify in $K$.
Then there exists an Euler system $c$ of rank $[K:\QQ]$ for $\ZZ_p(1)$ over $K$ such that both
\begin{itemize}
\item[(i)] $c$ interpolates the values of the Dedekind zeta function of $K$ at all positive even integers, and
\item[(ii)] $c$ determines all higher Fitting ideals of the Selmer group of $\mathbb{G}_m$ over abelian extensions of $K$.
\end{itemize}
\end{theorem*}

The interpolation property in claim (i) is stated precisely in Theorem \ref{main}(i) and asserts, roughly speaking, that for any positive even integer $j$ the `cyclotomic $j$-twist' of $c$ recovers the value of the $p$-truncated (that is, without Euler factors at $p$-adic places) Dedekind zeta function $\zeta_{K,\{p\}}(s)$ of $K$ at $s=j$. The proof of this result relies on an interpretation in terms of the functional equation for vertical determinantal systems of the fact that the Deligne-Ribet $p$-adic $L$-function interpolates $\zeta_{K,\{p\}}(1-j)$  and also requires us to prove the validity of the `local Tamagawa number conjecture' in certain new cases. Our main result in this regard is Theorem \ref{ltnc} and also has consequences for the validity of the Tamagawa number conjecture itself (see Theorem \ref{tnc evidence}).

Claim (ii) of Theorem 1 follows, essentially directly, from results of  \cite{sbA} and of Kurihara and the present authors in \cite{bks1}.

At this point, we should observe that Sakamoto \cite{sakamoto} has also recently constructed an Euler system of rank $[K:\QQ]$ for $\ZZ_p(1)$  over a totally real field $K$ and used it to give an equivariant generalization of the main result of Kurihara in \cite{kurihara}.  Sakamoto's construction also makes essential use of the Deligne-Ribet $p$-adic $L$-function 
but is otherwise quite different from ours. In fact, the construction of \cite{sakamoto} relies both on a non-canonical `rank reduction' technique for Euler systems (that does not use Coleman maps) and on a detailed technical analysis of Iwasawa-theoretic exterior power biduals that is used to define a generalization of the classical notion of characteristic ideal. 

In contrast, our construction follows as a rather formal consequence of the canonical functional equation for vertical determinantal systems and this difference of approach allows us to prove that the Euler system has natural interpolation properties and, at the same time, to avoid difficult auxiliary hypotheses such as the assumed vanishing of $\mu$-invariants and the need to project to suitable `components' of the cohomology groups of $\ZZ_p(1)$ over abelian extensions of $K$ (both of which seem to be essential to the approach of \cite{sakamoto}).


Further, in an attempt to extend the interpolation property in Theorem \ref{th1}(i) to odd positive integers we are led to formulate (in Conjecture \ref{pBC}) a variant of the `$p$-adic Beilinson conjecture' that is formulated by Besser, Buckingham, de Jeu and Roblot in \cite{BBJR}. Our conjecture is naturally formulated in terms of the  `generalized Stark elements' $\eta_{K}(j) $ introduced by Kurihara and the present authors in \cite{bks2-2} (see Definition \ref{gse def}). 
We note that $\eta_{K}(j)$ is defined in terms of the value of the $p$-truncated Dedekind zeta function of $K$ at $s=j$.

In this context, the construction of Theorem \ref{th1} can be  interpreted as a relation between the elements $\eta_{K}(1-j)$ and $\eta_{K}(j)$ for even positive integers $j$ and in Theorem \ref{CI} we prove a precise relation between $\eta_{K}(1-j)$ and $\eta_{K}(j)$ for odd positive $j$.

To give more details we set $r:=[K:\QQ]$, and for an odd integer $j>1$ note that $\eta_{K}(j)$ and $\eta_{K}(1-j)$ are respectively elements of $\CC_p$ and $ \CC_p\otimes_{\ZZ_p} {\bigwedge}_{\ZZ_p}^{r} H^1(\cO_{K}[1/p],\ZZ_p(j))$. We introduce a `higher rank Coates-Wiles homomorphism'
$$\Phi_j \in {\bigwedge}_{\QQ_p}^r H^1(K \otimes_\QQ \QQ_p,\QQ_p(j))$$
that is canonical up to sign, and in the case $K=\QQ$ coincides (up to sign) with $(p^{j-1}-1)$ times the classical Coates-Wiles homomorphism (see Definition \ref{defCW} and Remark \ref{remCW}). We write ${\rm loc}_p:  \CC_p\otimes_{\ZZ_p} {\bigwedge}_{\ZZ_p}^{r} H^1(\cO_{K}[1/p],\ZZ_p(j)) \to  \CC_p\otimes_{\ZZ_p} {\bigwedge}_{\ZZ_p}^{r} H^1(K \otimes_\QQ \QQ_p,\ZZ_p(j))$ for the map induced by the localization map at $p$ and $D_K$ for the discriminant of $K$.

\begin{theorem*}[Theorem \ref{CI}]\label{th2}
Let $K$ be a totally real field and $p$ an odd prime that does not ramify in $K$.
Then for each odd integer $j>1$ one has
$${\rm loc}_p(\eta_{K}(1-j))= \pm \eta_{K}(j)\cdot D_K^j \cdot \Phi_j \text{ in }\CC_p\otimes_{\ZZ_p} {\bigwedge}_{\ZZ_p}^r H^1(K \otimes_\QQ \QQ_p,\ZZ_p(j)).$$
\end{theorem*}

We find that this result leads to 
a natural conjectural generalization of the classical `Coleman-Ihara formula' in \cite{ihara} in the case that $K = \QQ$ (see Conjecture \ref{GCI} and Proposition \ref{CIP}). Such a link seems  striking and also appears to be new.

Finally, we note that in a future article, we will also formulate, and provide evidence for,  conjectural congruence relations between $\eta_{K}(j)$ and $\eta_{K}(k)$ for general integers $j$ and $k$ that are of a very different nature to the `functional equation' relations we focus on here.

The basic contents of the article is as follows. In \S\ref{sec3} and \S\ref{gsetn} we deal with preliminary material relating to Tamagawa number conjectures and generalized Stark elements and, in particular, construct a canonical higher rank Coleman map (see Remark \ref{higher col}). In \S\ref{fe vs} we introduce a natural notion of `local vertical determinantal systems' and use it to establish a functional equation for the vertical determinantal systems that are studied in \cite{sbA}. In \S\ref{construct section} we combine a special case of the main result of \S\ref{fe vs} together
 with results from \S\ref{sec3} and \S\ref{gsetn} in order to prove a precise version of Theorem \ref{th1}. Finally, in \S\ref{gci section} we prove Theorem \ref{th2} and discuss links between our approach, the $p$-adic Beilinson conjecture of Besser et al \cite{BBJR} and the Coleman-Ihara formula discussed in \cite{ihara}.

\subsection{General notation and convention} \label{notation}
\subsubsection{}For a commutative unital ring $R$ we write $D^{\rm perf}(R)$ for the derived category of perfect complexes of $R$-modules and $\det_R(-)$ for the determinant functor on $D^{\rm perf}(R)$ that was originally constructed by Knudsen and Mumford in \cite{KM} and later clarified by Knudsen in \cite{knudsen}. We note, however, that, whilst $\det_R(-)$ takes values in the Picard category of graded invertible $R$-modules we prefer, for convenience, to omit any explicit reference to gradings (since in the present context this does not lead to any confusion).

We denote the dual $\Hom_R(X,R)$ of an $R$-module $X$ by $X^\ast$. The functor ${\det}^{-1}_R(-)$ is then equal to ${\det}_R(-)^\ast$. For $X \in D^{\rm perf}(R)$, the canonical isomorphism
$${\det}_R(X) \otimes_R {\det}_R^{-1}(X) \xrightarrow{\sim} R; \ a\otimes f \mapsto f(a)$$
is used frequently and called the `evaluation map'.

For a non-negative integer $r$, the $r$-th exterior power bidual of an $R$-module $X$ is defined by
$${\bigcap}_R^r X:=\left({\bigwedge}_R^r (X^\ast)\right)^\ast.$$
For basic properties, see \cite[Appendix A]{sbA}.


For a field $E$ of characteristic $0$, we often abbreviate $E \otimes - $ to $E(-)$. Here $\otimes$ is one of $\otimes_\ZZ, \otimes_{\ZZ_p}, \otimes_\QQ, \otimes_{\QQ_p}$, depending on the context.

The following observation is used frequently: for a perfect complex $C$ of $\ZZ_p$-modules such that $C \lotimes_{\ZZ_p}\QQ_p$ is acyclic outside degree one, we have a canonical isomorphism
\begin{eqnarray}\label{qp}
\QQ_p\otimes_{\ZZ_p}{\det}_{\ZZ_p}^{-1}(C) \simeq {\bigwedge}_{\QQ_p}^d H^1(C \lotimes_{\ZZ_p}\QQ_p),
\end{eqnarray}
where $d:=\dim_{\QQ_p}(H^1(C\lotimes_{\ZZ_p}\QQ_p))$.

\subsubsection{}\label{fixed}
In the sequel we fix a number field $K$ and an algebraic closure $\overline K$ of $K$ and we set $G_K := \Gal(\overline K/K)$. We regard $\overline K \subset \CC$.

We fix an odd prime number $p$. We also fix an isomorphism $\CC \simeq \CC_p$.

For each natural number $n$ we write $\mu_{p^n}$ for the group of $p^n$-th roots of unity in $\overline K \subset \CC$. Note that we have a canonical generator $\zeta_{p^n}:=e^{2\pi \sqrt{-1}/p^n}$ of $\mu_{p^n}$, which determines a canonical generator $\xi:=(\zeta_{p^n})_n$ of $\ZZ_p(1)=\varprojlim_n \mu_{p^n}$.

For a finite set $S$ of places of $K$, we write $\zeta_{K,S}(s)$ for the $S$-truncated Dedekind zeta function of $K$. For any integer $j$, we denote its leading term at $s=j$ by $\zeta_{K,S}^\ast(j)$. We regard $\zeta_{K,S}^\ast(j) \in \CC_p$ via the fixed isomorphism $\CC \simeq \CC_p$.

We write $D_K$ for the discriminant of $K$, which is by definition
$$D_K:=\det({\rm Tr}_{K/\QQ}(x_i x_j)),$$
where $\{x_1,\ldots, x_r\}$ is a $\ZZ$-basis of $\cO_K$ (with $r:=[K:\QQ]$) and ${\rm Tr}_{K/\QQ}: K\to \QQ$ denotes the trace map.

We write $S_\infty(K)$ and $S_p(K)$ for the set of infinite and $p$-adic places of $K$. We set
$$K_p:=K\otimes_\QQ\QQ_p \simeq \prod_{v \in S_p(K)}K_v.$$

For any place $v$ of $K$, we fix an algebraic closure $\overline K_v$ of $K_v$ and an embedding $\overline K \hookrightarrow \overline K_v$. In particular, we regard $\mu_{p^n} \subset \overline K_v$ and $G_{K_v}:=\Gal(\overline K_v/K_v) \subset G_K$. If $v $ is an infinite place, we identify $\overline K_v=\CC$. 

For a finite place $v$ of $K$, we write ${\N}v$ for the cardinality of the residue field at $v$.

For a $\ZZ[1/2][\Gal(\CC/\RR)]$-module $X$, we set
$$X^+:=e^+ X,$$
where $e^+:=(1+c)/2$ with $c$ denoting the complex conjugation.

We use the standard notation for Galois (\'etale) cohomology. In particular, we write $\rgamma(\cO_{K,S},-)$, $\rgamma_c(\cO_{K,S},-)$ and $\rgamma_f(K,-)$ for $S$-cohomology, compactly supported $S$-cohomology and Bloch-Kato Selmer complexes respectively.


\subsubsection{}\label{labeling}

{\it Throughout this article, we fix a labelling
\[ \{\sigma_i: 1\le i\le [K:\QQ]\}\]
of the set of embeddings $\{\sigma: K \hookrightarrow \CC\}$. }This choice then determines an ordered $\QQ$-basis of the Betti cohomology space
$$H_K(j):=H_B^0(\Spec K(\CC),\QQ(j)) =\bigoplus_{\sigma: K\hookrightarrow \CC}(2\pi\sqrt{-1})^j \QQ$$
and ordered $\ZZ_p$-bases of both
$$\bigoplus_{\sigma: K\hookrightarrow \CC}\ZZ_p(j) $$
and
$$Y_K(j):=\bigoplus_{v \in S_\infty(K)}H^0(K_v,\ZZ_p(j))=\left(\bigoplus_{\sigma: K\hookrightarrow \CC}\ZZ_p(j)\right)^+$$
for any integer $j$. (For an explicit choice of bases, see \cite[\S 2.1]{bks2-2}.) Thus we have identifications
$${\det}_\QQ (H_K(j))=\QQ, \ {\det}_{\ZZ_p}\left(\bigoplus_{\sigma: K\hookrightarrow \CC}\ZZ_p(j) \right)= \ZZ_p \text{ and }{\det}_{\ZZ_p}(Y_K(j))=\ZZ_p.$$

Set $S:=S_\infty(K)\cup S_p(K)$. Then, since $p$ is odd, the Artin-Verdier Duality Theorem gives rise (via, for example, \cite[\S5, Lem. 12(b)]{BFetnc}) to an exact triangle in $D^{\rm perf}(\ZZ_p)$
\[\rgamma_c(\cO_{K,S},\ZZ_p(j))\to \DR\!\Hom_{\ZZ_p}(\rgamma(\cO_{K,S},\ZZ_p(1-j)),\ZZ_p[-3])\to Y_K(j)[0] \]
and hence to a canonical isomorphism
\begin{equation}\label{av isom} {\det}_{\ZZ_p}^{-1}(\rgamma_c(\cO_{K,S},\ZZ_p(j))) \simeq {\det}_{\ZZ_p}^{-1}(\rgamma(\cO_{K,S},\ZZ_p(1-j))) \otimes_{\ZZ_p} {\det}_{\ZZ_p}(Y_K(j)).\end{equation}
Taking account of the identification ${\det}_{\ZZ_p}(Y_K(j))=\ZZ_p$ above, we thereby obtain an isomorphism
\begin{eqnarray}\label{cisom}
{\det}_{\ZZ_p}^{-1}(\rgamma_c(\cO_{K,S},\ZZ_p(j))) \simeq {\det}_{\ZZ_p}^{-1}(\rgamma(\cO_{K,S},\ZZ_p(1-j)))
\end{eqnarray}
that will be used frequently in later sections.

\section{Coleman maps and local Tamagawa numbers}\label{sec3}

In this section we set $r:=[K:\QQ]$.


The aim of this section is to prove the `local Tamagawa number conjecture' when $p$ does not ramify in $K$ (see Theorem \ref{ltnc}). In the proof, we introduce a `higher rank Coates-Wiles homomorphism' in Definition \ref{defCW} and prove its interpolation property in Theorem \ref{colint}, which is used in later sections.


\subsection{The local Tamagawa number conjecture}

In this subsection, we formulate the local Tamagawa number conjecture for number fields (Conjecture \ref{localTNC}) and state the main result in this section (Theorem \ref{ltnc}).

For each integer $j$,
%
%
%
%
note that the decomposition $\CC=\RR(j) \oplus \RR(j-1)$ induces an isomorphism
\begin{eqnarray}\label{jj}
\RR  K \simeq  \RR H_K(j)^+ \oplus \RR H_K(j-1)^+.
\end{eqnarray}
(See \cite[(1)]{bks2-2}.) By composing this with the isomorphism
\begin{eqnarray}\label{decomp}
 \RR H_K(j)^+ \oplus \RR H_K(j-1)^+ \simeq \RR H_K(j); \ (x,y) \mapsto 2x + \frac{1}{2}\cdot \per \cdot  y,
 \end{eqnarray}
we obtain an isomorphism
\begin{eqnarray}\label{period}
\RR K \simeq \RR H_K(j).
\end{eqnarray}
The necessity of $2$ and $1/2$ in the definition of (\ref{decomp}) was observed by Kato \cite[p.8]{katolecture2} (see also \cite[(4) in \S 3.5.3]{FK}). Although we assume $p$ is odd and so difference by $2$-power is not essential, this modification is necessary to obtain a neat interpolation property of Coleman maps, which we prove in Theorem \ref{colint} below.

Also, note that there is a canonical isomorphism
\begin{eqnarray}\label{betti}
\bigoplus_{\sigma: K\hookrightarrow \CC}\QQ_p(j) \simeq \QQ_p H_K(j).
\end{eqnarray}

For any positive integer $j$, let
$$\log_{\QQ_p(j)}: H^1_f(K_p,\QQ_p(j)) \xrightarrow{\sim} K_p$$
denote the Bloch-Kato logarithm map. Note that $H^1(K_p,\QQ_p(j))=H^1_f(K_p,\QQ_p(j))$ if $j>1$.

\begin{definition}\label{loc def}
Let $j$ be a positive integer. When $j>1$, we define
$$\vartheta_j^{\rm loc} : \CC_p\otimes_{\ZZ_p}\left( {\det}_{\ZZ_p}^{-1}(\rgamma(K_p, \ZZ_p(j))) \otimes_{\ZZ_p} {\det}^{-1}_{\ZZ_p} \left( \bigoplus_{\sigma: K \hookrightarrow \CC} \ZZ_p(j)\right) \right) \xrightarrow{\sim} \CC_p$$
by the composition
\begin{eqnarray*}
 &&\CC_p\otimes_{\ZZ_p}\left( {\det}_{\ZZ_p}^{-1}(\rgamma(K_p, \ZZ_p(j))) \otimes_{\ZZ_p} {\det}^{-1}_{\ZZ_p} \left( \bigoplus_{\sigma :K \hookrightarrow \CC} \ZZ_p(j)\right) \right) \\
 &\stackrel{(\ref{qp}) \text{ and }(\ref{betti})}{\simeq}& \CC_p \otimes_{\QQ_p} \left({\bigwedge}_{\QQ_p}^r H^1(K_p,\QQ_p(j)) \otimes_{\QQ} {\bigwedge}_\QQ^r H_K(j)^\ast \right)\\
 &\stackrel{\log_{\QQ_p(j)}}{\simeq}&\CC_p \otimes_\QQ\left( {\bigwedge}_\QQ^r K \otimes_{\QQ}  {\bigwedge}_\QQ^r H_K(j)^\ast\right) \\
 &\stackrel{(\ref{period})}{\simeq}& \CC_p \otimes_\QQ\left( {\bigwedge}_\QQ^r H_K(j) \otimes_{\QQ}  {\bigwedge}_\QQ^r H_K(j)^\ast\right) \\
 &\stackrel{\text{evaluation}}{\simeq} &\CC_p.
\end{eqnarray*}
When $j=1$, we define $\vartheta_j^{\rm loc}$ in a similar way, by using the canonical exact sequence
$$0 \to H^1_f(K_p, \QQ_p(1)) \to H^1(K_p,\QQ_p(1)) \to \prod_{v\in S_p(K)}\QQ_p \to 0$$
and the canonical isomorphism
$$H^2(K_p, \QQ_p(1)) \simeq \prod_{v \in S_p(K)}\QQ_p.$$
\end{definition}

We now state the local Tamagawa number conjecture for the pair $(h^0(K)(j),p)$.

\begin{conjecture}[{${\rm TNC}_p^{\rm loc}(h^0(K)(j))$}]\label{localTNC}
For every positive integer $j$ there exists a (unique) $\ZZ_p$-basis
$$z_j \in {\det}_{\ZZ_p}^{-1}(\rgamma(K_p, \ZZ_p(j))) \otimes_{\ZZ_p} {\det}^{-1}_{\ZZ_p} \left( \bigoplus_{\sigma: K \hookrightarrow \CC} \ZZ_p(j)\right) $$
with the property that
$$\vartheta_j^{\rm loc}(z_j)=  \frac{\zeta_{K, \{p\}}^\ast(1-j)}{\zeta_{K,\{p\}}^\ast(j)} .$$
\end{conjecture}

We can now state the main result of this section.

\begin{theorem}\label{ltnc}
If $p$ does not ramify in $K$, then ${\rm TNC}_p^{\rm loc}(h^0(K)(j))$ is valid for every positive integer $j$.
\end{theorem}

The proof will be given in \S \ref{pf col}.

\begin{remark}\label{rem ltnc}
The key ingredients in the proof of Theorem \ref{ltnc} are the functional equation of $\zeta_K(s)$ and the classical explicit reciprocity law of Bloch and Kato \cite[Th. 2.1]{BK}. In fact, in the case $K=\QQ$ the validity of Theorem \ref{ltnc} is essentially verified in \cite[\S 6]{BK}, although the language used in loc. cit. is rather different from ours. For the same reason, the general case of Theorem \ref{ltnc} can presumably be derived (after suitable translation) from the sort of calculations that are made by Benois and Nguyen Quang Do in \cite{BN}. Nevertheless, we prefer to give a short direct proof of Theorem \ref{ltnc} both because the techniques it relies on will be used again in later sections and because, as far as we are aware, the verification of the precise form of Conjecture \ref{localTNC} in the case that $p$ is unramified in $K$ does not appear anywhere else in the literature.  (We note, however, in the special case that $K/\QQ$ is abelian, somewhat similar methods to ours are used in \cite{katolecture2}, \cite{BN}, \cite{HK} and \cite{BF2} to prove a natural $\Gal(K/\QQ)$-equivariant version of Conjecture \ref{localTNC}.) For direct consequences of Theorem \ref{ltnc}, see Theorem \ref{tnc evidence}.
\end{remark}

\subsection{Coleman maps}\label{coleman review}
In this subsection, we assume that $p$ is unramified in $K$.

The aim of this subsection is to construct a `higher rank Coleman map' $\Phi_{\bm{x}}$ in (\ref{phi def}) below (which is canonical up to sign). Our construction is motivated by Fukaya and Kato \cite[\S 3.6]{FK}.

We write $\mu_{p^\infty}$ for the union of $\mu_{p^n}$ over all $n$ and then set
$$ G_n := \Gal(K(\mu_{p^n})/K), \ G_\infty:=\Gal(K(\mu_{p^\infty})/K) \simeq \varprojlim_n G_n$$
and
$$\Lambda:=\ZZ_p[[G_\infty]].$$
For each $v \in S_p(K)$ we write $U_{K_v(\mu_{p^n})}^{(1)}$ for the group of principal units of $K_v(\mu_{p^n})$ and set
$$\cU_\infty:=\prod_{v \in S_p(K)}\varprojlim_n  U_{K_v(\mu_{p^n})}^{(1)} .$$
We also set
$$\cO_{K_p}:=\cO_K \otimes_\ZZ \ZZ_p \simeq \prod_{v\in S_p(K)}\cO_{K_v}.$$

We then recall (from either \cite[(4.7)]{BK} or the original article  \cite{coleman} of Coleman) that there exists a canonical exact sequence of $\Lambda$-modules
\begin{eqnarray}\label{colseq}0 \to \prod_{v \in S_p(K)} \ZZ_p(1) \to \cU_\infty \to \cO_{K_p}[[G_\infty]] \to \prod_{v \in S_p(K)} \ZZ_p(1) \to 0.
\end{eqnarray}
The middle map is called the Coleman map. Note that, to define the Coleman map, we need to choose a system of $p$-power roots of unity in $\overline K_v$ for each $v \in S_p(K)$ (which is used to characterize the Coleman power series). This choice is made by using the canonical basis $\xi$ of $H^0(\overline K, \ZZ_p(1))$ and the fixed embedding $\overline K \hookrightarrow \overline K_v$ (see \S \ref{fixed}).

Since $\Lambda$ is a regular ring, one can take ${\det}_\Lambda$ to any finitely generated $\Lambda$-modules and so the above sequence induces an isomorphism
\begin{eqnarray}
&&{\det}_\Lambda(\cU_\infty) \otimes_\Lambda {\det}_\Lambda^{-1}(\cO_{K_p}[[G_\infty]]) \label{colseq2}\\
  &\simeq& {\det}_\Lambda\left( \prod_{v \in S_p(K)} \ZZ_p(1)\right) \otimes_\Lambda {\det}_\Lambda^{-1} \left( \prod_{v \in S_p(K)} \ZZ_p(1)\right) \nonumber\\
&\simeq &\Lambda,\nonumber
\end{eqnarray}
where the last isomorphism is induced by the natural evaluation map.

We write $\Lambda^\#$ for the set $\Lambda$, regarded as endowed with its natural structure as a $\Lambda$-module and the action of $G_K$ that is given by setting $g(\lambda) := \lambda\cdot\overline{g}^{-1}$ for each $g \in G_K$ and $\lambda\in \Lambda$, where $\overline{g}$ denotes the image of $g$ in $G_\infty$.

Then Kummer theory gives a canonical exact sequence
$$ 0 \to \cU_\infty \to H^1(K_p,\Lambda^\#(1)) \to \prod_{v \in S_p(K)}\ZZ_p \to 0$$
and class field theory a canonical isomorphism
$$H^2(K_p, \Lambda^\#(1)) \simeq \prod_{v \in S_p(K)}\ZZ_p,$$
which combine to induce canonical isomorphism
$$ {\det}_\Lambda(\cU_\infty) \simeq {\det}_\Lambda^{-1}(\rgamma(K_p,\Lambda^\# (1))) \simeq {\det}_\Lambda(\rgamma(K_p,\Lambda^\# ))^\#, $$
where the last isomorphism is induced by the local duality. Here, by abuse of notation, for a $\Lambda$-module $X$ we denote by $X^\#$ the set $X$ on which $\Lambda$ acts via involution.
By combining this with (\ref{colseq2}), we obtain a canonical isomorphism
$$\Phi: {\det}_\Lambda(\rgamma(K_p,\Lambda^\# ))^\# \otimes_\Lambda {\det}_\Lambda^{-1}(\cO_{K_p}[[G_\infty]])  \simeq \Lambda.$$

{\it We now fix a $\ZZ$-basis $\bm{x}=\{x_1,\ldots,x_r\}$ of $\cO_K$}. This basis  determines an isomorphism of $\Lambda$-modules $ {\det}_\Lambda^{-1}(\cO_{K_p}[[G_\infty]])  \simeq \Lambda$ and hence also a composite isomorphism

$$\Phi_{\bm{x}}:  {\det}_\Lambda(\rgamma(K_p,\Lambda^\# ))^\# \simeq {\det}_\Lambda(\rgamma(K_p,\Lambda^\# ))^\# \otimes_\Lambda {\det}_\Lambda^{-1}(\cO_{K_p}[[G_\infty]])  \xrightarrow{\Phi} \Lambda.$$
We regard this isomorphism as an element
\begin{eqnarray}\label{phi def}
\Phi_{\bm{x}} \in  {\det}_\Lambda^{-1}(\rgamma(K_p,\Lambda^\# ))^\# \simeq \varprojlim_n  {\det}_{\ZZ_p[G_n]}^{-1}(\rgamma(K_p(\mu_{p^n}),\ZZ_p ))^\#.
\end{eqnarray}
Since $\bm{x}$ is a $\ZZ$-basis, $\Phi_{\bm{x}}$ is canonical up to sign.


\begin{remark}\label{higher col}
Set $S:=S_\infty(K)\cup S_p(K)$. Then in Theorem \ref{funceq} below we will show that there is a canonical one-to-one correspondence of the form
\begin{multline*}\{ \text{$\Lambda$-basis of }{\det}_\Lambda^{-1}(\rgamma(K_p,\Lambda^\# ))^\#\}\\ \simeq {\rm Isom}_{\Lambda}\left({\det}_\Lambda^{-1}(\rgamma_c(\cO_{K,S}, \Lambda^\#))^\#, {\det}_\Lambda^{-1}(\rgamma_c(\cO_{K,S}, \Lambda^\#(1)))\right).\end{multline*}
Via this correspondence, the element $\Phi_{\bm{x}}$ of $ {\det}_\Lambda^{-1}(\rgamma(K_p,\Lambda^\# ))^\#$ is regarded as an isomorphism
$$\Phi_{\bm{x}}: {\det}_\Lambda^{-1}(\rgamma_c(\cO_{K,S}, \Lambda^\#))^\# \xrightarrow{\sim} {\det}_\Lambda^{-1}(\rgamma_c(\cO_{K,S}, \Lambda^\#(1))).$$
We shall refer to this isomorphism as a `higher rank Coleman map' (and also see Remark \ref{modify} for a slight modification). Note also that, when $K$ is totally real, the Deligne-Ribet $p$-adic $L$-function is naturally regarded as a basis of $ {\det}_\Lambda^{-1}(\rgamma_c(\cO_{K,S}, \Lambda^\#(1)))^+$ (see Theorem \ref{IMC} for the detail).
\end{remark}

\subsection{The proof of Theorem \ref{ltnc}}\label{pf col}

We keep assuming that $p$ is unramified in $K$.

The case $j=1$ easily follows from the analytic class number formula and so, we may, and will, assume in the sequel that $j>1$.

Let $\chi_{\rm cyc}: G_\infty \xrightarrow{\sim} \ZZ_p^\times$ denote the cyclotomic character. The $j$-th power of $\chi_{\rm cyc}$ induces a surjection
$$\chi_{\rm cyc}^{j}: {\det}_\Lambda^{-1} (\rgamma(K_p, \Lambda^\#))^\# \to {\det}_{\ZZ_p}^{-1}(\rgamma(K_p,\ZZ_p(j))).$$
More precisely, it is the composition
\begin{multline}\label{twist map}
{\det}_\Lambda^{-1} (\rgamma(K_p, \Lambda^\#))^\# \xrightarrow{a \mapsto a \otimes \xi^{\otimes j}} {\det}_\Lambda^{-1} (\rgamma(K_p, \Lambda^\#))^\#\otimes_\Lambda \ZZ_p(j) \simeq  {\det}_{\ZZ_p}^{-1}(\rgamma(K_p,\ZZ_p(j))),
\end{multline}
where $\xi \in \ZZ_p(1)$ is the canonical basis (see \S \ref{fixed}) and the last isomorphism follows from \cite[Prop. 1.6.5(3)]{FK}.

For the later use, we give the following definition.
\begin{definition} \label{defCW}
Let $\Phi_{\bm{x}} \in {\det}_\Lambda^{-1} (\rgamma(K_p, \Lambda^\#))^\#$ be the basis constructed in (\ref{phi def}).
For any $j>1$, we define a {\it higher rank Coates-Wiles homomorphism} by
$$\Phi_j:= \chi_{\rm cyc}^{j}(\Phi_{\bm{x}}) \in {\det}^{-1}_{\ZZ_p}(\rgamma(K_p,\ZZ_p(j))) .$$
This is a $\ZZ_p$-basis by definition.
\end{definition}

\begin{remark}\label{remCW}
When $K=\QQ$, one checks by definition that
$$\Phi_j=\pm (p^{j-1}-1)\varphi_j^{\rm CW} \text{ in } H^1(\QQ_p,\QQ_p(j)) \stackrel{(\ref{qp})}{\simeq} {\det}_{\QQ_p}^{-1}(\rgamma(\QQ_p,\QQ_p(j))),$$
where $\varphi_j^{\rm CW} \in H^1(\QQ_p,\QQ_p(j))$ is the classical Coates-Wiles homomorphism defined in \cite[\S 2]{BK}.
In the general case $\Phi_j$ can be explicitly described as follows. First, note that we have a natural identification
$$H^1(K_p,\QQ_p(j))=\Hom_{\Lambda}(\cU_{\infty},\QQ_p(j)).$$
(See \cite[p. 342]{BK}.) Let
$${\rm Col}: \cU_\infty \to \cO_{K_p}[[G_\infty]]$$
be the Coleman map (i.e., the middle map in (\ref{colseq})). By taking the functor $\Hom_{\Lambda}(-,\QQ_p(j))$ to this map, we obtain a map
$${\rm Col}_j: \Hom_{\QQ_p}(K_p, \QQ_p(j)) \stackrel{f \mapsto f \circ \chi_{\rm cyc}^j}{\simeq} \Hom_{\Lambda}(\cO_{K_p}[[G_\infty]], \QQ_p(j)) \to \Hom_{\Lambda}(\cU_{\infty},\QQ_p(j)).$$
(One deduces from the exact sequence (\ref{colseq}) that this is in fact an isomorphism.) Since $\QQ_p(j)$ has a canonical basis $\xi^{\otimes j}$ (see \S \ref{fixed}), we can identify $\Hom_{\QQ_p}(K_p,\QQ_p(j))$ with $K_p^\ast:=\Hom_{\QQ_p}(K_p,\QQ_p)$. Let $\bm{x}=\{x_1,\ldots,x_r\}$ be the fixed $\ZZ$-basis of $\cO_K$, and $\{x_1^\ast ,\ldots, x_r^\ast\}$ be the dual basis. One sees that the image of $x_1^\ast \wedge \cdots \wedge x_r^\ast$ under the map
$${\bigwedge}_{\QQ_p}^r K_p^\ast \xrightarrow{{\rm Col}_j} {\bigwedge}_{\QQ_p}^r \Hom_{\Lambda}(\cU_{\infty},\QQ_p(j))={\bigwedge}_{\QQ_p}^r H^1(K_p,\QQ_p(j)) \stackrel{(\ref{qp})}{\simeq} {\det}_{\QQ_p}^{-1}(\rgamma(K_p,\QQ_p(j)))$$
coincides with $\Phi_j$.
\end{remark}

We now study interpolation properties of higher rank Coates-Wiles homomorphisms.
We use the identification
$${\det}^{-1}_{\ZZ_p} \left( \bigoplus_{\sigma: K \hookrightarrow \CC} \ZZ_p(j)\right) =\ZZ_p  $$
(see \S \ref{labeling}) to regard $\vartheta_j^{\rm loc}$ in Definition \ref{loc def} as a map
$$\vartheta_j^{\rm loc}: \CC_p {\det}^{-1}_{\ZZ_p}(\rgamma(K_p,\ZZ_p(j))) \xrightarrow{\sim} \CC_p.$$

We also note that the discriminant $D_K$ of $K$ belongs to $\ZZ_p^\times$ since $p$ is assumed to be unramified in $K$.

Theorem \ref{ltnc} is thereby reduced to the following result.

\begin{theorem}\label{colint}
We have
$$\vartheta_j^{\rm loc}(\Phi_j) =\pm  D_K^{-j} \cdot \frac{\zeta_{K, \{p\}}^\ast(1-j)}{\zeta_{K,\{p\}}^\ast(j)} .$$
\end{theorem}

\begin{remark}\label{modify}
Since $D_K \in \ZZ_p^\times$, there exists $\sigma_{D_K} \in G_\infty$ such that $\chi_{\rm cyc}(\sigma_{D_K})=D_K$. Then Theorem \ref{colint} implies the element $\Psi:=\sigma_{D_K}\cdot \Phi_{\bm{x}}$ satisfies the  interpolation property
$$\vartheta_j^{\rm loc}(\chi_{\rm cyc}^j (\Psi))= \pm \frac{\zeta_{K, \{p\}}^\ast(1-j)}{\zeta_{K,\{p\}}^\ast(j)} \text{ for any $j>1$}.$$
This shows that the collection 
$\{\pm \zeta_{K,\{p\}}^\ast (1-j)/\zeta_{K,\{p\}}^\ast (j) \}_{j>1}$ is $p$-adically interpolated by $\Psi$ and so is considerably stronger than the statement of Theorem \ref{ltnc}. In particular, Theorem \ref{th1}(i) in the introduction will later be proved by applying Theorem \ref{colint}.
\end{remark}

The rest of this section is devoted to the proof of Theorem \ref{colint}. In the sequel we will write $r_\RR$ and $r_\CC$ for the number of real and complex places of $K$ respectively (so that $r = r_\RR + 2r_\CC$).

\begin{lemma}\label{fe0}
For any positive integer $j$, we have
$$\frac{\zeta_K^\ast(1-j)}{\zeta_K^\ast(j)} =\begin{cases}
 \displaystyle
\pm 2^{r_\RR}(2\pi)^{r_\CC}\left( \frac{(j-1)!}{(2\pi)^j}\right)^r  |D_K|^{j-\frac 12} &\text{if $j$ is even,}\\
 \displaystyle \pm 2^{-r_\RR} (2\pi)^{r_\RR+r_\CC} \left( \frac{(j-1)!}{(2\pi)^j}\right)^r| D_K|^{j-\frac 12}&\text{if $j$ is odd.}
  \end{cases}$$
\end{lemma}

\begin{proof}
This follows from the well-known functional equation
$$\zeta_K(1-s)= |D_K|^{s-\frac 12} \left( \cos \frac{\pi s}{2}\right)^{r_\RR+r_\CC} \left(\sin \frac{\pi s}{2}\right)^{r_\CC } \left( 2 (2\pi)^{-s}\Gamma(s)\right)^r \zeta_K(s).$$
\end{proof}

Recall that we fixed a $\ZZ$-basis $\{x_1,\ldots,x_r\}$ of $\cO_K$.

\begin{lemma}\label{compute}
Let
$$\alpha_j: \RR {\bigwedge}_\QQ^r K \simeq \RR {\bigwedge}_\QQ^r H_K(j)$$
be the isomorphism induced by (\ref{period}). We identify $\ {\bigwedge}_\QQ^r H_K(j)=\QQ$ (see \S \ref{labeling}). Then we have
$$\alpha_j(x_1\wedge \cdots \wedge x_r)= \begin{cases}
\displaystyle \pm 2^{r_\RR}(2\pi)^{r_\CC} \frac{\sqrt{|D_K|}}{(2\pi)^{jr}}  &\text{if $j$ is even},\\
\displaystyle \pm 2^{-r_\RR}(2\pi)^{r_\RR+r_\CC} \frac{\sqrt{|D_K|}}{(2\pi)^{jr}} &\text{if $j$ is odd}.
\end{cases}$$

\end{lemma}

\begin{proof}
This follows from a simple computation.
\end{proof}

By combining Lemmas \ref{fe0} and \ref{compute}, we have the following.

\begin{corollary}\label{cor comp}
We have
$$\alpha_j \left( \left(\prod_{v\in S_p(K)} \frac{1-{\N}v^{j-1}}{1-{\N}v^{-j}} \right) (j-1)!^r D_K^{j-1} \cdot x_1\wedge \cdots \wedge x_r \right) = \pm \frac{\zeta_{K, \{p\}}^\ast(1-j)}{\zeta_{K, \{p\}}^\ast(j)} .$$
\end{corollary}

Theorem \ref{colint} follows from Corollary \ref{cor comp} and the following key lemma, which is essentially the explicit reciprocity law due to Bloch and Kato \cite[Th. 2.1]{BK}.

\begin{lemma}\label{exp rec}
Let
$$\log_{\QQ_p(j)}:{\det}^{-1}_{\QQ_p}(\rgamma(K_p,\QQ_p(j))) \stackrel{(\ref{qp})}{\simeq} {\bigwedge}_{\QQ_p}^r H^1(K_p, \QQ_p(j)) \simeq {\bigwedge}_{\QQ_p}^r K_p $$
be the isomorphism induced by the Bloch-Kato logarithm map. Then we have
$$\log_{\QQ_p(j)} (\Phi_j) =   \left(\prod_{v\in S_p(K)} \frac{1-{\N}v^{j-1}}{1-{\N}v^{-j}} \right) (j-1)!^r D_K^{-1} \cdot x_1\wedge \cdots \wedge x_r. $$

\end{lemma}

\begin{proof}
Let $\exp_{\QQ_p(j)}$ denote the inverse of $\log_{\QQ_p(j)}$.
Let
$$\Phi_j^\ast: {\bigwedge}_{\QQ_p}^r H^1(K_p, \QQ_p(j)) \xrightarrow{\sim} {\bigwedge}_{\QQ_p}^r K_p^\ast$$
be the isomorphism defined by $\Phi_j \mapsto x_1^\ast \wedge \cdots \wedge x_r^\ast$. By Remark \ref{remCW}, this map coincides with the map induced by `$c\circ b \circ a$' in \cite[p. 367]{BK} ($r$ in loc. cit. corresponds to our $j$). So, by \cite[Claim 4.8]{BK}, the composition
$${\bigwedge}_{\QQ_p}^r K_p \xrightarrow{\exp_{\QQ_p(j)}} {\bigwedge}_{\QQ_p}^r H^1(K_p, \QQ_p(j)) \xrightarrow{\Phi_j^\ast} {\bigwedge}_{\QQ_p}^r K_p^\ast$$
coincides with the map induced by
$$K_p \to K_p^\ast; \ x \mapsto \left( y \mapsto (j-1)!^{-1} {\rm Tr}_{K/\QQ} ((1-p^{-j}{\rm Fr}_p)(x) \cdot (1-p^{j-1}{\rm Fr}_p)^{-1} (y))\right).$$
Here ${\rm Fr}_p:=({\rm Fr}_v)_{v \in S_p(K)}$ is the automorphism of $K_p=\prod_{v\in S_p(K)}K_v$ determined by the Frobenius element ${\rm Fr}_v \in \Gal(K_v/\QQ_p)$ for each $v \in S_p(K)$.
By Lemma \ref{det frob} below, the determinant of this map (with respect to the bases $\{x_1,\ldots,x_r\}$ and $\{x_1^\ast,\ldots,x_r^\ast\}$) is
$$ (j-1)!^{-r} D_K \prod_{v \in S_p(K)} \frac{1- {\N}v^{-j}}{1-{\N}v^{j-1}}. $$
Thus we have
$$\exp_{\QQ_p(j)} (x_1\wedge\cdots \wedge x_r)=\left( (j-1)!^{-r} D_K \prod_{v \in S_p(K)} \frac{1- {\N}v^{-j}}{1-{\N}v^{j-1}}\right)\cdot \Phi_j. $$
This proves the lemma.
\end{proof}

\begin{lemma}\label{det frob}
Let $E/\QQ_p$ be a finite cyclic extension of degree $f$. Then for any generator $\sigma \in \Gal(E/\QQ_p)$ we have an equality of polynomials
$$\det(1-\sigma \cdot  t \mid E) = 1-t^f.$$
($E$ is regarded as a $\QQ_p$-vector space.)
\end{lemma}
\begin{proof}
The matrix of $\sigma$ with respect to a normal basis of $E/\QQ_p$ is
$$\begin{pmatrix}
0 & 0& \cdots  &\cdots & 1 \\
1 & 0 &\cdots & \cdots &0 \\
0 & 1 & 0 & \cdots &0 \\
\vdots && \ddots&\ddots&\vdots \\
0 & \cdots& \cdots& 1 & 0
\end{pmatrix}.$$
Thus we have
$$\det(1-\sigma \cdot  t \mid E)=\det\begin{pmatrix}
1 & 0& \cdots  &\cdots & -t \\
-t & 1 &\cdots & \cdots &0 \\
0 & -t & 1 & \cdots &0 \\
\vdots && \ddots&\ddots&\vdots \\
0 & \cdots& \cdots& -t & 1
\end{pmatrix} = 1-t^f.$$
\end{proof}

\section{Generalized Stark elements and Tamagawa numbers}\label{gsetn}


In this section, we give a review of generalized Stark elements introduced in \cite{bks2-2} (see Definition \ref{gse def}). We also formulate the Tamagawa number conjecture for number fields (see Conjecture \ref{tnc conj}) and, by applying Theorem \ref{ltnc}, we give some new evidence for the conjecture (see Theorem \ref{tnc evidence}).

Let $K$ be a number field and $p$ an odd prime number. We set
$$S:=S_\infty(K) \cup S_p(K).$$

\subsection{Period-regulator isomorphisms}\label{period section}

For any integer $j$, one can define a canonical `period-regulator isomorphism'
$$\vartheta_{\ZZ_p(j)}: \CC_p{\det}_{\ZZ_p}^{-1}(\rgamma_c(\cO_{K,S},\ZZ_p(j)))  \xrightarrow{\sim} \CC_p.$$
(More generally, for a general $p$-adic representation $T$ coming from geometry, and any finite set $\Sigma$ of places of $K$ that contains $S_\infty(K) \cup S_p(K)$ and all places at which $T$ ramify, one can define a canonical isomorphism $\vartheta_{T}: \CC_p{\det}_{\ZZ_p}^{-1}(\rgamma_c(\cO_{K,\Sigma},T))  \xrightarrow{\sim} \CC_p$.)

For the reader's convenience we now review the explicit definition of $\vartheta_{\ZZ_p(j)}$ in the case that $j\neq 0,1$. To do this we set
%
%
%
$$r_j:={\rm rank}_{\ZZ_p}(Y_K(j))={\rm rank}_\QQ (H_K(j)^+) =\begin{cases}
r_\RR+r_\CC &\text{if $j$ is even},\\
r_\CC&\text{if $j$ is odd}.
\end{cases}$$
For each $j > 1$ we also write
\begin{equation}\label{borel reg iso} {\rm reg}_j: \RR K_{2j-1}(\cO_K) \simeq \RR H_K(j-1)^+\end{equation}
for the canonical Borel regulator isomorphism. It is known that the Chern character map induces an isomorphism
\begin{equation}\label{chern iso}{\rm ch}_j: \QQ_p K_{2j-1}(\cO_K) \simeq H^1(\cO_{K,S},\QQ_p(j))\end{equation}
(cf. \cite[\S 2.2.1]{bks2-2}).

\subsubsection{The case $j<0$}\label{neg pr map section}

In this case the space $H^2(\cO_{K,S},\QQ_p(1-j))$ vanishes (by Soul\'e \cite[Th. 10.3.27]{NeuSW}). We can therefore define $\vartheta_{\ZZ_p(j)}$ to be the composition

\begin{eqnarray*}
\CC_p{\det}_{\ZZ_p}^{-1}(\rgamma_c(\cO_{K,S},\ZZ_p(j))) &\stackrel{(\ref{av isom})}{\simeq}& \CC_p\left({\det}_{\ZZ_p}^{-1}(\rgamma(\cO_{K,S},\ZZ_p(1-j))) \otimes_{\ZZ_p} {\det}_{\ZZ_p}^{-1}(Y_K(-j)) \right) \\
&\stackrel{(\ref{qp})}{\simeq} &\CC_p \left({\bigwedge}_{\ZZ_p}^{r_j} H^1(\cO_{K,S},\ZZ_p(1-j)) \otimes_{\ZZ_p} {\bigwedge}_{\ZZ_p}^{r_j} Y_K(-j)^\ast \right) \\
&\stackrel{{\rm ch}_{1-j}}{\simeq}& \CC_p\left( {\bigwedge}_\ZZ^{r_j} K_{1-2j}(\cO_K) \otimes_{\ZZ}  {\bigwedge}_{\ZZ_p}^{r_j} Y_K(-j)^\ast\right) \\
&\stackrel{{\rm reg}_{1-j}}{\simeq}&\CC_p \left( {\bigwedge}_{\QQ}^{r_j}H_K(-j)^+ \otimes_{\ZZ}  {\bigwedge}_{\ZZ_p}^{r_j} Y_K(-j)^\ast  \right)\\
&\simeq &\CC_p,
\end{eqnarray*}
where the last isomorphism uses the canonical identification $\QQ_p H_K(-j)^+ = \QQ_p Y_K(-j)$ (see (\ref{betti})).

\subsubsection{The case $j>1$}
In this case we define $\vartheta_{\ZZ_p(j)}$ to be the composition

\begin{eqnarray*}
&&\CC_p {\det}_{\ZZ_p}^{-1}(\rgamma_c(\cO_{K,S}, \ZZ_p(j))) \\
&\simeq& \CC_p \left({\det}_{\QQ_p}^{-1}(\rgamma_f(K,\QQ_p(j))) \otimes_{\QQ_p} {\det}_{\QQ_p}(\rgamma_f(K_p,\QQ_p(j)))   \otimes_{\ZZ_p} {\rm det}_{\ZZ_p}(Y_K(j)) \right)\\
&\simeq& \CC_p \left({\det}_\ZZ(K_{2j-1}(\cO_K))  \otimes_{\ZZ} {\det}^{-1}_{\QQ_p}(H^1(K_p,\QQ_p(j)))  \otimes_{\ZZ_p} {\rm det}_{\ZZ_p}(Y_K(j))\right) \\
&\stackrel{{\rm reg}_j \otimes \log_{\QQ_p(j)}}{\simeq}&  \CC_p \left({\det}_\QQ(H_K(j-1)^+) \otimes_\QQ {\det}^{-1}_\QQ(K)  \otimes_{\ZZ} {\rm det}_{\ZZ_p}(Y_K(j))\right) \\
&\stackrel{(\ref{jj})}{\simeq}& \CC_p\left({\det}_\QQ(H_K(j-1)^+) \otimes_\QQ {\det}_\QQ^{-1}(H_K(j-1)^+) \right) \\
&\simeq &\CC_p,
\end{eqnarray*}
where the first isomorphism is obtained from the exact triangle
$$\rgamma_c(\cO_{K,S},\QQ_p(j)) \to \rgamma_f(K,\QQ_p(j)) \to \rgamma_f(K_p, \QQ_p(j)) \oplus \rgamma (K \otimes_\QQ \RR, \QQ_p(j)),$$
and for the second isomorphism we used the fact that $H^1_f(K,\QQ_p(j)) \simeq H^1(\cO_{K,S},\QQ_p(j)) \stackrel{{\rm ch}_j}{\simeq} \QQ_p K_{2j-1}(\cO_K)$.

\begin{remark}\label{sch rem} It is conjectured by Schneider in \cite{ps} that the space $H^2(\cO_{K,S},\QQ_p(1-j))$ vanishes for all $j > 1$. In \cite[Th. 1.8(3)]{DJN} it is shown by de Jeu and Navilarekallu that this conjecture is equivalent to asserting that for any $j > 1$ the Deligne-Ribet $p$-adic $L$-function $L_p(\omega^{1-j},s)$ should not vanish at  $s=j$. For any integer $j$ for which this prediction is valid, the approach of \cite[\S 2.2.4]{bks2-2} shows that there is a canonical exact sequence
\begin{eqnarray}\label{sch short}
0 \to H^1(\cO_{K,S}, \QQ_p(1-j)) \to H^1(K_p,\QQ_p(j))^\ast \to \QQ_p K_{2j-1}(\cO_K)^\ast \to 0.
\end{eqnarray}
Using this sequence, it can be shown that the isomorphism
\begin{multline*} \CC_p {\det}_{\ZZ_p}^{-1}(\rgamma_c(\cO_{K,S}, \ZZ_p(j)))\\ \simeq \CC_p \left({\det}_\ZZ(K_{2j-1}(\cO_K))  \otimes_{\ZZ} {\det}^{-1}_{\QQ_p}(H^1(K_p,\QQ_p(j)))  \otimes_{\ZZ_p} {\rm det}_{\ZZ_p}(Y_K(j))\right)\end{multline*}
obtained above coincides with the composition
\begin{eqnarray*}
& &\CC_p{\det}_{\ZZ_p}^{-1}(\rgamma_c(\cO_{K,S},\ZZ_p(j)))\\
 &\stackrel{(\ref{av isom})}{\simeq}& \CC_p\left({\det}_{\ZZ_p}^{-1}(\rgamma(\cO_{K,S},\ZZ_p(1-j))) \otimes_{\ZZ_p} {\det}_{\ZZ_p}(Y_K(j))\right) \\
&\stackrel{(\ref{qp})}{\simeq} &\CC_p\left( {\bigwedge}_{\ZZ_p}^{r_j} H^1(\cO_{K,S},\ZZ_p(1-j)) \otimes_{\ZZ_p} {\bigwedge}_{\ZZ_p}^{r_j} Y_K(j) \right)\\
&\stackrel{(\ref{sch short})}{\simeq}& \CC_p\left( {\bigwedge}_\ZZ^{r-r_j} K_{2j-1}(\cO_K) \otimes_\ZZ {\bigwedge}_{\QQ_p}^r H^1(K_p,\QQ_p(j))^\ast \otimes_{\ZZ_p} {\bigwedge}_{\ZZ_p}^{r_j} Y_K(j) \right),
\end{eqnarray*}
where $r:=[K:\QQ]$.
\end{remark}

\subsubsection{The Tamagawa number conjecture}

We can now give a precise statement of the Tamagawa number conjecture for the pair $(h^0(K)(j),p)$.

\begin{conjecture}[{${\rm TNC}_p(h^0(K)(j))$}]\label{tnc conj}
Let $K$ be a number field and $j$ an integer. Then there exists a (unique) $\ZZ_p$-basis
$$z_{K,S} \in {\det}_{\ZZ_p}^{-1}(\rgamma_c(\cO_{K,S},\ZZ_p(j)))$$
with the property that
$$\vartheta_{\ZZ_p(j)}(z_{K,S})=\zeta_{K,S}^\ast(j).$$
\end{conjecture}


In the sequel we shall make use of the following (well-known) compatibility result concerning this conjecture.

\begin{proposition}\label{fe compatible} If ${\rm TNC}_p^{\rm loc}(h^0(K)(j))$ (Conjecture \ref{localTNC}) is valid, then the assertions of ${\rm TNC}_p(h^0(K)(j))$ and ${\rm TNC}_p(h^0(K)(1-j))$ are equivalent.
\end{proposition}

\begin{proof} This can be deduced as a special case of \cite[Th. 5.3]{BFetnc} and can also be derived more directly as follows.

Upon comparing the explicit statements of ${\rm TNC}_p(h^0(K)(j))$ and ${\rm TNC}_p(h^0(K)(1-j))$ with that of ${\rm TNC}_p^{\rm loc}(h^0(K)(j))$, the given claim is reduced to proving the following:  if one sets
$$D_j^{\rm loc}:={\det}_{\ZZ_p}^{-1}(\rgamma(K_p, \ZZ_p(j))) \otimes_{\ZZ_p} {\det}^{-1}_{\ZZ_p} \left( \bigoplus_{\sigma: K \hookrightarrow \CC} \ZZ_p(j)\right)$$
then the diagram
\begin{equation}\label{fe}\displaystyle
\xymatrix{
\CC_p\left( D_j^{\rm loc}\otimes_{\ZZ_p} {\det}_{\ZZ_p}^{-1}(\rgamma_c(\cO_{K,S},\ZZ_p(j)))  \right)\ar[r]^{\quad\quad\sim} \ar[d]_{\vartheta_j^{\rm loc}\otimes \vartheta_{\ZZ_p(j)}}& \CC_p {\det}_{\ZZ_p}^{-1}(\rgamma_c(\cO_{K,S},\ZZ_p(1-j)))  \ar[d]^{\vartheta_{\ZZ_p(1-j)}}\\
\CC_p \otimes_{\CC_p} \CC_p \ar[r]_{a\otimes b \mapsto ab} & \CC_p
}
\end{equation}
commutes. The isomorphism in the upper row of this diagram is induced by the Artin-Verdier Duality Theorem via (\ref{av isom}), the explicit relation between $\rgamma_c(\cO_{K,S},\ZZ_p(j))$ and $\rgamma(\cO_{K,S},\ZZ_p(j))$ that follows from the definition of compactly supported cohomology and (\ref{decomp}). 
%
%
The commutativity of the diagram can then be checked by means of an explicit comparison of the definitions of  $\vartheta_j^{\rm loc}, \vartheta_{\ZZ_p(j)}$ and $\vartheta_{\ZZ_p(1-j)}$. \end{proof}

\subsection{Generalized Stark elements}\label{gse sec}

In this section we fix an integer $j$ with $j \notin \{0,1\}$ and review the definition of the generalized Stark elements from \cite{bks2-2}.

\begin{definition}\label{gse def} If $j$ is either negative, or both positive and such that  $H^2(\cO_{K,S},\QQ_p(1-j))$ vanishes (cf. \S\ref{neg pr map section} and Remark \ref{sch rem}), then there is a canonical identification of $\CC_p$-spaces
\begin{align*}\CC_p {\det}^{-1}_{\ZZ_p}(\rgamma_c(\cO_{K,S}, \ZZ_p(j)))\stackrel{(\ref{cisom})}{\simeq}\, &\CC_p {\det}^{-1}_{\ZZ_p}(\rgamma(\cO_{K,S}, \ZZ_p(1-j))) \\ \stackrel{(\ref{qp})}{\simeq}\, &\CC_p {\bigwedge}_{\ZZ_p}^{r_j} H^1(\cO_{K,S},\ZZ_p(1-j)).\end{align*} 

In any such case the generalized Stark element
$$\eta_{K,S}(j)\in \CC_p {\bigwedge}_{\ZZ_p}^{r_j} H^1(\cO_{K,S},\ZZ_p(1-j)) $$
is defined to be the image  of $\vartheta_{\ZZ_p(j)}^{-1}(\zeta_{K,S}^\ast(j)) \in \CC_p {\det}^{-1}_{\ZZ_p}(\rgamma_c(\cO_{K,S}, \ZZ_p(j)))$ under this isomorphism, where $\vartheta_{\ZZ_p(j)}$ is the period-regulator isomorphism defined in \S\ref{period section}.

In the case that $j$ is positive and $H^2(\cO_{K,S},\QQ_p(1-j))$ does not vanish, we set $\eta_{K,S}(j):=0$.
\end{definition}

\begin{remark} \label{explicit stark}
If $K$ is totally real and $j$ is odd, then $\eta_{K,S}(j)$ belongs to $\CC_p$ and is explicitly described as follows.
\begin{itemize}
\item[(i)] If $j<0$, then $\eta_{K,S}(j)=\zeta_{K,S}(j)$ and so belongs to $\QQ$ (by a  well-known result of Klingen and Siegel).
\item[(ii)] If $j > 1$, then $\eta_{K,S}(j)=\det(\kappa)\cdot \zeta_{K,S}(j)$, where $\kappa$ is the composite homomorphism
    $$\CC_p K_{2j-1}(\cO_K) \xrightarrow{{\rm syn}^j_p} \CC_p K \stackrel{(\ref{jj})}{\simeq} \CC_p H_K(j-1) \stackrel{{\rm reg}_j^{-1}}{\simeq} \CC_p K_{2j-1}(\cO_K).$$
Here ${\rm syn}^j_p$ is induced by the syntomic regulator maps $K_{2j-1}(\cO_{K_v}) \to K_v$ for each $v$ in $S_p(K)$ and is bijective if and only if $H^2(\cO_{K,S},\QQ_p(1-j))$ vanishes.
\end{itemize}
\end{remark}

\begin{remark}\label{even remark}
If $K$ is totally real and $j$ is positive even, then the isomorphisms (\ref{borel reg iso}) and (\ref{chern iso}) combine to imply that $H^1(\cO_{K,S},\QQ_p(j))\simeq \QQ_p K_{2j-1}(\cO_K)$ vanishes and this in turn implies $H^2(\cO_{K,S},\QQ_p(1-j))$ vanishes. In this case, therefore, the generalized Stark element $\eta_{K,S}(j)$ is non-trivial. In addition, whilst $\eta_{K,S}(j)$ is, a priori, an element of $\CC_p {\bigwedge}_{\ZZ_p}^{[K:\QQ]} H^1(\cO_{K,S},\ZZ_p(1-j))$, Theorem \ref{tnc evidence}(i) below implies that if $p$ does not ramify in $K$, then $\eta_{K,S}(j)$ is `rational' in the following sense:
$$\eta_{K,S}(j) \in \QQ_p {\bigwedge}_{\ZZ_p}^{[K:\QQ]} H^1(\cO_{K,S},\ZZ_p(1-j)). $$
\end{remark}

\subsection{Deligne-Ribet $p$-adic $L$-functions}\label{DR section}
In this subsection, let $K$ be a totally real field.

We set
$$ \Lambda:=\ZZ_p[[\Gal(K(\mu_{p^\infty})^+/K)]]$$
%
and write $Q(\Lambda)$ for the total quotient ring of $\Lambda$.

We fix a topological generator $\gamma$ of $\Gal(K(\mu_{p^\infty})^+/K(\mu_p)^+)$ and recall that there exists a canonical element
$$\cL \in \frac{1}{\gamma-1} \cdot \Lambda \subset Q(\Lambda)$$
(the `Deligne-Ribet $p$-adic $L$-function') that satisfies
$$\chi_{\rm cyc}^j (\cL)=\zeta_{K,S}(1-j) \text{ for any positive even integer $j$}.$$

\begin{theorem}[Iwasawa main conjecture] \label{IMC}
The complex $\rgamma_c(\cO_{K,S}, \Lambda^\#(1)) \lotimes_\Lambda Q(\Lambda)$ is acyclic and there exists a (unique) $\Lambda$-basis
$$\fz \in {\det}_{\Lambda}^{-1}(\rgamma_c(\cO_{K,S}, \Lambda^\#(1) ))$$
such that the natural map
$${\det}_\Lambda^{-1} (\rgamma_c(\cO_{K,S}, \Lambda^\#(1))) \hookrightarrow {\det}_\Lambda^{-1}(\rgamma_c(\cO_{K,S}, \Lambda^\#(1))) \otimes_\Lambda Q(\Lambda)  =Q(\Lambda)$$
sends $\fz$ to $\cL$.
\end{theorem}

\begin{proof} It is well-known that the complex $\rgamma_c(\cO_{K,S}, \Lambda^\#(1))$ is acyclic outside degrees $2$ and $3$ and such that
$$H^i_c(\cO_{K,S}, \Lambda^\#(1)) = \begin{cases}
\fX_K:=\Gal(M_K/ K(\mu_{p^\infty})^+) &\text{ if $i=2$,}\\
\ZZ_p &\text{ if $i=3$,}
\end{cases}$$
where $M_K$ denotes the maximal abelian pro-$p$ extension of $K(\mu_{p^\infty})^+$ that is unramified outside $S$. In particular, since $\fX_K$ is known to be a torsion $\Lambda$-module, these descriptions imply that $\rgamma_c(\cO_{K,S}, \Lambda^\#(1)) \lotimes_\Lambda Q(\Lambda)$ is acyclic and also that the image of the natural composite map
\begin{equation*}\label{char}
 {\det}_\Lambda^{-1} (\rgamma_c(\cO_{K,S}, \Lambda^\#(1))) \hookrightarrow {\det}_\Lambda^{-1}(\rgamma_c(\cO_{K,S}, \Lambda^\#(1))) \otimes_\Lambda Q(\Lambda) = Q(\Lambda)\end{equation*}
is equal to
\[ {\det}_\Lambda^{-1}(\fX_K[-2])\cdot {\det}_\Lambda^{-1}(\ZZ_p[-3]) = {\rm char}_\Lambda(\fX_K) \cdot {\rm char}_\Lambda(\ZZ_p)^{-1}.
\]

The claimed result therefore follows directly from the fact that
$$\Lambda\cdot \cL =  {\rm char}_\Lambda(\fX_K) \cdot {\rm char}_\Lambda(\ZZ_p)^{-1}, $$
as proved by Wiles in \cite[Th. 1.3 and 1.4]{wiles}.
\end{proof}

\subsection{Evidence for the Tamagawa number conjecture}

Theorem \ref{ltnc} leads to the following evidence in support of Conjecture \ref{tnc conj}.

\begin{theorem}\label{tnc evidence} If $p$ is odd and unramified in $K$, then ${\rm TNC}_p(h^0(K)(j))$ is valid in all of the following cases.

\begin{itemize}
\item[(i)] $K$ is totally real and $j$ is positive even.
\item[(ii)] $K$ is imaginary quadratic and $p$ is bigger than $3$ and splits in $K$.
\end{itemize}
\end{theorem}

\begin{proof}

At the outset we note that, for any odd prime $p$, Theorem \ref{IMC} implies directly that ${\rm TNC}_p(h^0(K)(j))$ is valid for any negative odd integer $j$, whilst Theorem \ref{ltnc} implies that ${\rm TNC}^{\rm loc}_p(h^0(K)(j))$ is valid for all $j$ provided that $p$ does not ramify in $K$.

In particular, if we assume that $p$ does not ramify in $K$, then claim (i) follows directly upon combining Proposition \ref{fe compatible} with Theorems \ref{ltnc} and \ref{IMC}.

In a similar way, claim (ii) is reduced to the main result of Johnson-Leung \cite{JL} which asserts that, under the given hypotheses, ${\rm TNC}_p(h^0(K)(1-j))$ is valid for all positive integers $j$. (The case $j=1$ follows from the class number formula.)
\end{proof}

%
%
%

To consider the validity of ${\rm TNC}_p(h^0(K)(j))$ for integers $j$ that are both odd and bigger than one we first reformulate the  $p$-adic Beilinson conjecture in terms of generalized Stark elements (cf. Remark \ref{explicit stark}(ii)).

\begin{conjecture}[The $p$-adic Beilinson conjecture]\label{pBC}
Let $K$ be a totally real field. Then for each odd integer $j>1$ one has
$$\chi_{\rm cyc}^{1-j}(\cL)=\eta_{K,S}(j).$$
\end{conjecture}

\begin{remark}\label{abel}
With the usual notation of the $p$-adic $L$-function, we have
$$\chi^{1-j}_{\rm cyc}(\cL)=L_p(\omega^{1-j},j).$$
By using this fact one can show that 
Conjecture \ref{pBC} is equivalent to the conjecture \cite[Conj. 2.17(2)]{BBJR} of Besser et al. In particular, \cite[Rem. 4.18]{BBJR} implies that Conjecture \ref{pBC} is valid when $K$ is abelian over $\QQ$.\end{remark}

The next result (is not used in later sections and) follows immediately from Theorem \ref{IMC}.

\begin{proposition}\label{cor pbc}
Let $j>1$ be an odd integer such that $H^2(\cO_{K,S},\QQ_p(1-j))$ vanishes. Then Conjecture \ref{pBC} implies the validity of ${\rm TNC}_p(h^0(K)(j))$.
\end{proposition}

\begin{remark} A result of the same form as
Proposition \ref{cor pbc} has also recently been proved, and in a more general context, by Nickel in \cite{nickel}.
\end{remark}

\section{The functional equation of vertical determinantal systems}\label{fe vs}

In this section we shall further develop the theory of `vertical determinantal systems' that was introduced in \cite{sbA}. In particular, the context of the results proved here is much more general than those in earlier sections.

Then, in \S\ref{construct section}, a special case of the main result (Theorem \ref{funceq}) of this section will be combined with Theorems \ref{colint} and \ref{IMC} in order to construct a higher rank Euler system with canonical interpolation properties.

\subsection{Local vertical determinantal sytems} In this subsection we specify the notation and hypotheses that are to be used throughout \S\ref{fe vs} and then present a natural `local' analogue of a key definition from \cite{sbA}.

\subsubsection{}Let $K$ be a number field. We fix a representation $T$ of $G_K$ that is unramified outside a finite set $S_{\rm ram}(T)$ of places of $K$ and assume that $T$ is a free module with respect to an action of a commutative Gorenstein $\ZZ_p$-order $\mathcal{R}$ (that commutes with the given action of $G_K$).

We assume that $\cR$ is endowed with an involution and, for any profinite group $\cG$ and an $\cR[[\cG]]$-module $X$, we write $X^\#$ for the set $X$ on which $\cR[[\cG]]$ acts via involution.


We also fix an abelian extension $\cK$ of $K$ (that is not necessarily a pro-$p$ extension) and write $\Omega(\cK)$ for the collection of finite  extensions of $K$ in $\cK$.

For each $F$ in $\Omega(\cK)$ we set 
$$\cG_F := \Gal(F/K)$$
and write $T_F$ for the induced representation ${\rm Ind}_{G_F}^{G_K}(T)\simeq T\otimes_{\ZZ_p}\ZZ_p[\cG_F]^\#$.

We also consider the finite sets of places of $K$ that are defined by
\[ S  := S_\infty(K)\cup S_p(K)\cup S_{\rm ram}(T)\]
and
\[ S(F) := S \cup S_{\rm ram}(F/K), \]
where $S_{\rm ram}(F/K)$ denotes the (finite) set of places of $K$ that ramify in $F$.

For each set of places $\Sigma$ of $K$ we also write $\Sigma_f$ for its subset $\Sigma \setminus S_\infty(K)$ of finite places.
%


\subsubsection{}We recall that the module of vertical determinantal systems for the pair $(T,\cK)$ is defined as an inverse limit
$${\rm VS}(T,\cK):=\varprojlim_{F \in \Omega(\cK)} {\det}_{\cR[\cG_F]}^{-1}(\rgamma_c(\cO_{F,S(F)}, T^\ast(1))),$$
where the transition morphisms are specified in \cite[Def. 2.9]{sbA}  
and will be recalled in the proof of Theorem \ref{funceq} below.


In the following definition we present a natural `local' analogue of this construction. 

\begin{definition}\label{lvs def}
We define the module of {\it local vertical determinantal systems} for $(T,\cK)$ by setting
$${\rm VS}^{\rm loc}(T,\cK):=\varprojlim_{F \in \Omega(\cK)}\left(\bigotimes_{v \in S_f} {\det}^{-1}_{\cR[\cG_F]} (\rgamma(K_v, T_F))\right) \otimes_{\cR[\cG_F]} {\det}_{\cR[\cG_F]}^{-1}\left( \bigoplus_{F \hookrightarrow \CC} T\right).$$
Here the limit is defined with respect to the transition maps for each  $F$ and $F'$ in $\Omega(\cK)$ with $F \subset F'$ that are induced by the natural isomorphisms
\[ \cR[\cG_F]\otimes^{\DL}_{\cR[\cG_{F'}]}\rgamma(K_v, T_{F'})\simeq \rgamma(K_v, T_F)\]
in $D^{\rm perf}(\cR[\cG_F])$ for each $v$ in $S_f$ together with the natural isomorphism of $\cR[\cG_F]$-modules
\[\cR[\cG_F]\otimes_{\cR[\cG_{F'}]}\left(\bigoplus_{F' \hookrightarrow \CC} T\right) \simeq \bigoplus_{F \hookrightarrow \CC} T.\]

\end{definition}

\subsection{The functional equation} We can now state our main result concerning vertical determinant systems.

This result can be interpreted as showing that local vertical determinant systems give rise to  `functional equation' relations between  the vertical determinant systems that are associated to a representation and to its Kummer dual.

\begin{theorem}\label{funceq}
There exists a canonical isomorphism of $\cR[[\Gal(\cK/K)]]$-modules
$${\rm VS}^{\rm loc} (T,\cK) \simeq \Hom_{\cR[[\Gal(\cK/K)]]} \left( {\rm VS}(T^\ast(1),\cK), {\rm VS}(T,\cK)^\# \right).$$
\end{theorem}

The proof of this result will occupy the rest of this section.

\subsubsection{}We first fix some convenient notation for the argument.

For each $F$ in $\Omega(\cK)$ we abbreviate the determinant functor ${\rm det}_{\cR[\cG_F]}(-)$ to ${\rm d}_F(-)$. We then also set
\[\Xi_F(T) := {\rm d}_{F}^{-1}(\rgamma_c(\cO_{F,S(F)}, T^\ast(1)))\]
and
\[ \Xi^{\rm loc}_F(T) := \left(\bigotimes_{v \in S_f} {\rm d}^{-1}_{F} (\rgamma(K_v, T_F))\right) \otimes_{\cR[\cG_F]} {\rm d}_{F}^{-1}\left( \bigoplus_{F \hookrightarrow \CC} T\right).\]
%

For each place $v$ of $K$ outside $S$ we fix a place $w$ of $F$ above $v$ and write $F_w^\infty$ and $K_v^\infty$ for the unramified $\ZZ_p$-extensions of $F_w$ and $K_v$ respectively. We denote the finite group $ \Gal(F_w^\infty/K_v^\infty)$ by $\Delta_w$, set $\Gamma_w := \Gal(F_w^\infty/K_v)$ and note that the quotient group $\Gamma_v := \Gamma_w/\Delta_w\simeq \Gal(K_v^\infty/K_v)$ is generated by the restriction of the inverse 
$$\phi_v:={\rm Fr}_v^{-1}$$
of the Frobenius automorphism at $v$.

We write $\Lambda_w^\infty$ for the completed group algebra
$\ZZ_p[[\Gamma_w]]$. Then for any $\Lambda_w^\infty$-module $M$ the completed tensor product $M_w^\infty := \Lambda_w^\infty\hat\otimes_{\ZZ_p[\Delta_w]}M$ is a $\Lambda_w^\infty$-module via left multiplication upon which $\phi_v$ induces a well-defined endomorphism (that we also denote by $\phi_v$) that sends each element $x\otimes m$ to $x\hat\phi_v^{-1}\otimes \hat\phi_v(m)$, where $\hat\phi_v$ is any element of $\Gamma_w$ whose projection to $\Gamma_v$ coincides with that of $\phi_v$.

In particular, since $T$ is unramified at $v$ the restriction of $T_F$ to $G_{K_v}$ is a $\Lambda_w^\infty$-module and, in this case, there exists a natural short exact sequence of $\cR[\cG_F][[G_{K_v}]]$-modules
\[ 0 \to (T_F)_w^\infty \xrightarrow{ 1- \phi_v} (T_F)_w^\infty \to T_F\to 0\]
(for a proof of exactness see, for example, \cite[Prop. 2.2, Rem. 2.3]{sv}) and hence a canonical exact triangle in $D(\cR[\cG_F])$ of the form
\begin{equation}\label{unram tri}  \rgamma(K_v, (T_F)_w^\infty) \xrightarrow{1-\phi_v} \rgamma(K_v, (T_F)_w^\infty) \to \rgamma(K_v, T_F) \to .  \end{equation}
By Lemma \ref{technical lemma} below, this triangle belongs to the subcategory $D^{\rm perf}(\cR[\cG_F])$ of $D(\cR[\cG_F])$ and so induces a composite isomorphism of $\cR[\cG_F]$-modules
\begin{align*} \mu_v(T_F): {\rm d}_{F} (\rgamma(K_v, T_F)) \simeq\, &{\rm d}_{F}(\rgamma(K_v, (T_F)_w^\infty))\otimes_{\cR[\cG_F]}{\rm d}_{F}^{-1}  (\rgamma(K_v, (T_F)_w^\infty))\\
 \simeq\, &\cR[\cG_F],\end{align*}
in which the second isomorphism is the evaluation map on ${\rm d}_{F} (\rgamma(K_v, (T_F)_w^\infty))$.

\subsubsection{}Turning now to the construction of the claimed isomorphism, we recall that the Artin-Verdier Duality Theorem induces (via, for example, \cite[\S5, Lem. 14]{BFetnc}) a canonical isomorphism of $\cR[\cG_F]$-modules
\[ \theta^{\rm AV}(T_F): \Xi_F^{\rm loc}(T)\otimes_{\cR[\cG_F]}\bigotimes_{v \in S(F)\setminus S} {\rm d}_{F}^{-1} (\rgamma(K_v,T_F))\simeq \Xi_F(T^*(1))^{-1}\otimes_{\cR[\cG_F]}\Xi_F(T)^\#.\]

\noindent{}We can therefore define a canonical composite isomorphism
\begin{align*} \theta(T_F): \Xi_F^{\rm loc}(T) \simeq\, &\Xi_F^{\rm loc}(T)\otimes_{\cR[\cG_F]}\bigotimes_{v \in S(F)\setminus S} {\rm d}_{F}^{-1} (\rgamma(K_v,T_F))\\
\simeq\, &\Xi_F(T^*(1))^{-1}\otimes_{\cR[\cG_F]}\Xi_F(T)^\#\\
\simeq\, &\Hom_{\cR[\cG_F]}(\Xi_F(T^*(1)), \Xi_F(T)^\#),\end{align*}
in which the first map is induced by the tensor product over $v$ in $ S(F)\setminus S$ of the isomorphisms $\mu_v(T_F)$, the second is $\theta^{\rm AV}(T_F)$ and the third is the natural isomorphism.

To construct an isomorphism of the required sort it is therefore enough to prove that for each pair of fields $F$ and $F'$ in $\Omega(\cK)$ with $F\subset F'$ the following diagram commutes
\begin{equation}\label{needed diagram} \begin{CD} \Xi_{F'}^{\rm loc}(T) @> \theta(T_{F'}) >> \Hom_{\cR[\cG_{F'}]}(\Xi_{F'}(T^*(1)), \Xi_{F'}(T)^\#)\\
@V VV @V VV\\
\Xi_{F}^{\rm loc}(T) @> \theta(T_{F}) >> \Hom_{\cR[\cG_F]}(\Xi_F(T^*(1)), \Xi_F(T)^\#).
\end{CD}\end{equation}
Here the left hand vertical arrow is the transition morphism specified in Definition \ref{lvs def} and the right hand vertical arrow is the composite
\begin{align*} &\Hom_{\cR[\cG_{F'}]}(\Xi_{F'}(T^*(1)), \Xi_{F'}(T)^\#) \\ \to \, &\cR[\cG_{F}]\otimes_{\cR[\cG_{F'}]}\Hom_{\cR[\cG_{F'}]}(\Xi_{F'}(T^*(1)), \Xi_{F'}(T)^\#)\\
\simeq \, &\Hom_{\cR[\cG_{F}]}(\cR[\cG_{F}]\otimes_{\cR[\cG_{F'}]}\Xi_{F'}(T^*(1)), \cR[\cG_{F}]\otimes_{\cR[\cG_{F'}]}\Xi_{F'}(T)^\#)\\
\simeq \, &\Hom_{\cR[\cG_{F}]}(\Xi_{F}(T^*(1)), \Xi_{F}(T)^\#)\end{align*}

\noindent{}in which the first map is the natural projection, the second is the natural isomorphism resulting from the fact that $\Xi_{F'}(T^*(1))$ and $\Xi_{F'}(T)$ are both free $\cR[\cG_{F'}]$-modules of rank one and the third is induced by the isomorphism (for both $T$ and $T^*(1)$)
\[ \cR[\cG_{F}]\otimes_{\cR[\cG_{F'}]}\Xi_{F'}(T)\simeq \Xi_F(T)\]
that is induced by the transition morphism $\theta_{F'/F}(T):\Xi_{F'}(T) \to  \Xi_{F}(T)$ involved in the definition of ${\rm VS}(T,\cK)$.

In order to recall the definition of $\theta_{F'/F}(T)$, for each place $v$ of $K$ outside $S(F)$ we write $\kappa_v$ for its residue field, with separable closure $\overline \kappa_v$, and observe that the definition of $S(F)$ ensures the action of $G_{K_v}$ on $T_F$ factors through the restriction map $G_{K_v}\to \Gal(\overline \kappa_v/\kappa_v)$. In addition, the  complex $\rgamma(\kappa_v,T_F):= \rgamma(\Gal(\overline \kappa_v/\kappa_v),T_F)$ is represented by $T_F \xrightarrow{1-\phi_v} T_F$, where the first term is placed in degree zero and so the evaluation map on ${\rm d}_{F} (T_F)$ induces a canonical isomorphism of $\cR[\cG_F]$-modules
\[ \epsilon_v(T_F): {\rm d}_{F} (\rgamma(\kappa_v,T_F)) \simeq {\rm d}_{F} (T_F)\otimes_{\cR[\cG_F]}{\rm d}_{F}^{-1} (T_F) \simeq \cR[\cG_F].\]
%

Then, in terms of this notation, the transition map $\theta_{F'/F}(T^\ast(1))$ is defined (in \cite[\S2.4]{sbA}) to be the composite surjective homomorphism of $\cR[\cG_F]$-modules

\begin{eqnarray}
 {\rm d}_{F'}^{-1}(\rgamma_c(\cO_{F',S(F')}, T))  &\to&  {\rm d}_{F}^{-1}(\rgamma_c(\cO_{F,S(F')}, T)) \nonumber\\
 &\simeq&  {\rm d}_{F}^{-1}(\rgamma_c(\cO_{F,S(F)}, T)) \otimes \bigotimes_{v\in S(F')\setminus S(F)} {\rm d}_{F}(\rgamma(\kappa_v, T_F)) \nonumber\\
 &\simeq&  {\rm d}_{F}^{-1}(\rgamma_c(\cO_{F,S(F)}, T)). \nonumber
 \end{eqnarray}
Here the first map is induced by the canonical descent isomorphism in  $D^{\rm perf}(\cR[\cG_F])$
\[ \cR[\cG_F]\otimes^{\DL}_{\cR[\cG_{F'}]}\rgamma_c(\cO_{F',S(F')}, T) \simeq \rgamma_c(\cO_{F,S(F')}, T),\]
the second by the canonical exact triangle in $D^{\rm perf}(\cR[\cG_F])$
 \begin{equation*}\label{compact scalar change} \rgamma_{c}(\mathcal{O}_{F, S(F')},T) \to \rgamma_{c}(\mathcal{O}_{F,S(F)},T) \to \bigoplus_{v \in S(F')\setminus S(F)} \rgamma(\kappa_v, T_F)\to \end{equation*}
 (cf. \cite[Chap. II, Prop. 2.3d)]{med}) and the third by the isomorphisms $\epsilon_v(T_F)$ for each $v$ in $S(F')\setminus S(F)$. (To obtain the definition of $\theta_{F'/F}(T)$ one need only replace $T$ by $T^\ast(1)$ in this description.)

\subsubsection{}Now, since the Artin-Verdier Duality Theorem behaves functorially with respect to change of fields, the precise difference between the isomorphisms $\cR[\cG_F]\otimes_{\cR[\cG_{F'}]}\theta(T_{F'})$ and $\theta(T_F)$ is that the former involves $\mu_v(T_F)$ for each place $v$ in
\[ (S(F')\setminus S)\setminus ( S(F)\setminus S) = S(F')\setminus S(F) = S_{\rm ram}(F'/K)\setminus S(F)\]
whereas these maps do not occur in the definition of $\theta(T_F)$. In addition, the transition maps
$\theta_{F'/F}(T)$ and $\theta_{F'/F}(T^*(1))$ that occur in the right hand vertical arrow of (\ref{needed diagram}) involve the isomorphisms $\epsilon_v(T_F)$ and $\epsilon_v(T_F^*(1))$ for each such $v$, whilst these maps do not occur in the left hand vertical arrow of the diagram.


To proceed we write $T^\dagger$ for the representation $(T^*(1))^\ast \simeq T(-1)$ and note that for any place $v$ outside $S(F)$ the complex $\rgamma(\kappa_v,T^*(1)_F)^*[-2]$ is naturally isomorphic to $\rgamma(\kappa_v,T^\dagger_F)[-1]$. The local duality theorem therefore gives rise to a canonical exact triangle in $D^{\rm perf}(\cR[\cG_F])$ of the form
\begin{equation}\label{local duality} \rgamma(\kappa_v,T_F) \to \rgamma(K_v,T_F) \to \rgamma(\kappa_v,T^\dagger_F)[-1] \to .\end{equation}
The above observations then imply that the commutativity of (\ref{needed diagram}) will follow if for every $v$ in  $S(F')\setminus S(F)$ the following diagram of isomorphisms commutes
\begin{equation}\label{needed diagram2}\begin{CD}
{\rm d}_{F} (\rgamma(K_v, T_F)) @> \Delta_v(T_F) >> {\rm d}_{F} (\rgamma(\kappa_v, T_F))\otimes_{\cR[\cG_F]}{\rm d}_{F}^{-1} (\rgamma(\kappa_v, T^\dagger_F))\\
@V \mu_v(T_F)VV @VV \epsilon_v(T_F)\otimes \epsilon_v(T^\dagger_F)^{-1}V\\
\cR[\cG_F] @=  \cR[\cG_F]\otimes_{\cR[\cG_F]}\cR[\cG_F], \end{CD}\end{equation}
where $\Delta_v(T_F)$ is induced by the triangle (\ref{local duality}).

%
%

To verify this we shall use the following commutative diagram of exact triangles

\begin{equation}\label{needed diagram3}\begin{CD} \rgamma(\kappa_v, (T_F)_w^\infty) @> 1-\phi_v>>  \rgamma(\kappa_v, (T_F)_w^\infty) @> >> \rgamma(\kappa_v, T_F) @> >> \\
@V VV @V VV @V VV\\
\rgamma(K_v, (T_F)_w^\infty) @> 1-\phi_v >> \rgamma(K_v, (T_F)_w^\infty) @> >>  \rgamma(K_v, T_F) @> >> \\
@V VV @V VV @V VV \\
 \rgamma(\kappa_v, (T^\dagger_F)_w^\infty)[-1] @> 1-\phi_v >> \rgamma(\kappa_v, (T^\dagger_F)_w^\infty)[-1]  @> >> \rgamma(\kappa_v, T^\dagger_F)[-1]  @> >>  \\
@V VV @V VV @V VV\end{CD}\end{equation}
Here the second row is (\ref{unram tri}) and the first and third rows are constructed similarly. In addition, the third column is (\ref{local duality}) and the first and second columns are constructed in the same way. Finally, we note that the same argument as in Lemma \ref{technical lemma} below shows that the complexes $\rgamma(\kappa_v, (T_F)_w^\infty)$ and $\rgamma(\kappa_v, (T^\dagger_F)_w^\infty)$ belong to $D^{\rm perf}(\cR[\cG_F])$ and so therefore does the entire diagram.

The diagram (\ref{needed diagram3}) thus gives rise to a commutative diagram of $\cR[\cG_F]$-modules
\[\begin{CD}
{\rm d}_{F} (\rgamma(K_v, T_F)) @> \Delta_v(T_F) >> {\rm d}_{F} (\rgamma(\kappa_v, T_F))\otimes_{\cR[\cG_F]}{\rm d}_{F}^{-1} (\rgamma(\kappa_v, T^\dagger_F))\\
@V \mu_v(T_F)VV @VV \mu'_v(T_F)\otimes \mu'_v(T^\dagger_F)^{-1}V\\
\cR[\cG_F] @=  \cR[\cG_F]\otimes_{\cR[\cG_F]}\cR[\cG_F], \end{CD}\]
in which $ \mu'_v(T_F)$ denotes the isomorphism of $\cR[\cG_F]$-modules ${\rm d}_{F} (\rgamma(\kappa_v, T_F)) \simeq \cR[\cG_F]$  that is defined in the same way as $\mu_v(T)$ but with the role of the triangle (\ref{unram tri}) now played by the first  row of (\ref{needed diagram3}), and $\mu'_v(T^\dagger_F)$ is defined similarly using the third row of (\ref{needed diagram3}). To deduce the commutativity of (\ref{needed diagram2}) it is therefore enough to show
that, if $M$ denotes either $T$ or $T^\dagger$, then there is an equality of maps $\mu_v'(M_F) = \epsilon_v(M_F)$.

If $E_v$ is any finite unramified extension of $K_v$, with residue field $\tilde\kappa$, then it suffices to prove this equality after replacing $M$ by ${\rm Ind}_{G_{E_v}}^{G_{K_v}}(M)$ and $\cR[\cG_F]$ by $\cR[\cG_F\times \Gal(E_v/K_v)]$. Hence, by  choosing a suitable $E_v$ with $[E_v:K_v]$ prime to $p$ and applying Shapiro's Lemma, 
 we can (replace $\rgamma(K_v, -), \rgamma(\kappa_v, -)$  and $\phi_v$ by $\rgamma(E_v,-), \rgamma(\tilde\kappa,-)$ and $\phi_v^{[E_v:K_v]}$ and so) assume $\phi_v$ acts unipotently on $M$ modulo the radical of $\cR$. The key point now is that, since the $G_{K_v}$-module
$(M_F)_w^\infty$ is isomorphic to $\ZZ_p[[\Gamma_v]]\otimes_{\ZZ_p}M_F$, the complex $\rgamma(\kappa_v, (M_F)_w^\infty)$ can be computed as the inverse limit (over natural numbers $n$) of the commutative diagrams 
\[ \begin{CD} M_F @> 1-\phi_v^{p^n} >> M_F\\
@V \sum_{i=0}^{i = p-1}(\phi_v^{p^{n-1}})^iVV @\vert\\
M_F @> 1-\phi^{p^{n-1}}_v >> M_F\end{CD}\] 
and so is, under the present hypotheses on $\phi_v$, isomorphic in $D(\cR[\cG_F])$ to the complex that is equal to $M_F$ in degree one and to zero in all other degrees.

 Given this explicit description, the equality $\mu_v'(M_F) = \epsilon_v(M_F)$ then follows from a straightforward  comparison of the definitions of the respective maps.

This completes the proof of Theorem \ref{funceq}.

\begin{lemma}\label{technical lemma} For each $v$ in $S(F)\setminus S$ the complex $\rgamma(K_v, (T_F)_w^\infty)$ belongs to $D^{\rm perf}(\cR[\cG_F])$.\end{lemma}

\begin{proof} 

The definition of $T_F$ as an induced module implies that for any Galois extension $L$ of $K$ in $F$ there exists a natural isomorphism in $D(\cR[\cG_L])$ of the form
\[ \cR[\cG_{L}]\otimes^{\DL}_{\cR[\cG_F]}\rgamma(K_v, (T_F)_w^\infty)\simeq \rgamma(K_v, (T_L)_w^\infty).\]
By a standard argument (as used, for example, in the proof of \cite[Prop. 1.6.5(i)]{FK}), one is therefore reduced to showing that the cohomology groups of each complex $\rgamma(K_v, (T_L)_v^\infty)$ are finitely generated $\ZZ_p$-modules that vanish in almost all degrees.

Then, by an application of the Hochschild-Serre spectral sequence (relative to the finite extension $F_w/K_v$) it is enough to prove the cohomology groups of $\rgamma(F_w, (T_L)_w^\infty)$ are finitely generated $\ZZ_p$-modules that vanish in almost all degrees. Since the $G_{F_w}$-module $(T_L)_w^\infty$ is isomorphic to the direct sum of a finite number of copies of
$T^\infty := \ZZ_p[[\Gal(F_w^\infty/F_w)]]\otimes_{\ZZ_p}T$ it is therefore enough to prove that the same is true for the complex $\rgamma(F_w, T^\infty)$.

For each natural number $n$ we write $F_w^n$ for the unramified extension of $F_w$ of degree $p^n$. Then in each degree $i$ there is a natural isomorphism of groups 
\[ H^i(F_w, T^\infty)\cong \varprojlim_n\varprojlim_m H^i(F_w^n, T/p^m)\]
where, in the inverse limits, the transition morphisms over $m$ are induced by the projection maps $T/p^m \to T/p^{m-1}$ and the transition morphisms over $n$ by the natural corestriction maps.

In particular, since each module $H^i(F_w^n, T/p^m)$ is finite and (as $p$ is odd) vanishes if $i > 2$, the module $H^i(F_w, T^\infty)$ is compact and vanishes if $i > 2$. Since $H^i(F_w, T^\infty)/p$ embeds into $H^i(F_w, (T/p)^\infty)$, Nakayama's Lemma therefore reduces us to showing that $H^i(F_w, (T/p)^\infty)$ is finite for each $i \le 2$.

To do this we write $E$ for the finite (Galois) extension of $F_w$ that corresponds to the kernel of the action of $G_{F_w}$ on $(T/p)\oplus \ZZ_p(1)/p$ (so that $E$ contains a primitive $p$-th root of unity). Then, just as above, the Hochschild-Serre spectral sequence  
 allows us to replace $F_w$ by $E$ and hence assume both that $E$ contains a primitive $p$-th root of unity and that $T/p$ is equal to $ \ZZ_p(1)/p$. 
 
 It is thus enough to prove that in this case the group $H^i(E, (\ZZ_p(1)/p)^\infty)$ is finite for $i \le 2$. This is obvious if $i = 0$ and in the remaining cases can be deduced from the fact that $H^i(E, \ZZ_p(1)^\infty)$ is canonically isomorphic to $\ZZ_p$ if $i = 2$ and to $\ZZ_p(1)$ if $i =1$ (cf. \cite[Th. 11.2.4(iii)]{NeuSW}).
\end{proof}

\section{Construction of a higher rank Euler system}\label{construct section}

In this section we shall combine the relevant special case of Theorem \ref{funceq} with Theorems \ref{colint} and \ref{IMC} in order to deduce the existence of higher rank Euler systems for $\ZZ_p(1)$ over totally real fields that have canonical interpolation properties (see Theorem \ref{main}).

\subsection{The Euler system}\label{es section}

\subsubsection{}\label{general construct}We use the notations in \S \ref{fe vs}. The key fact we require is that vertical determinantal systems give rise to Euler systems for $(T,\cK)$ if the following hypothesis is satisfied.

\begin{hypothesis} \label{hyp1}\
\begin{itemize}
\item[(i)] $Y_K(T^\ast(1)):= \bigoplus_{v \in S_\infty(K)} H^0(K_v,T^\ast(1))$ is a free $\cR$-module.
\item[(ii)] $H^1(\cO_{F,S(F)},T)$ is $\ZZ_p$-free for every $F \in \Omega(\cK)$.
\item[(iii)] $H^0(F,T)=0$ for every $F \in \Omega(\cK)$.
\item[(iv)] All infinite places of $K$ split completely in $\cK$.
\end{itemize}
\end{hypothesis}

More precisely, if we assume this hypothesis and set
$$ r(T):={\rm rank}_\cR(Y_K(T^\ast(1))),$$
then \cite[Th. 2.18]{sbA} implies the existence of a homomorphism of $\cR[[\Gal(\cK/K)]]$-modules
$$\theta_{T,\cK}: {\rm VS}(T,\cK)^\# \to {\rm ES}_{r(T)}(T,\cK) \subset \prod_{F \in \Omega(\cK)}{\bigcap}^{r(T)}_{\cR[\cG_F]} H^1(\cO_{F,S(F)},T),$$
where ${\rm ES}_{r(T)}(T,\cK)$ is the module of Euler systems of rank $r(T)$ for the pair $(T,\cK)$. (See \cite[\S 6.1]{bss} for the definition. However, note that we do not assume $\cK/K$ is a $p$-extension.) This homomorphism is canonical up to a choice of an ordered basis of $Y_K(T^\ast(1))$.

\subsubsection{}We now fix a totally real field $K$ and specialize to the case $\cR=\ZZ_p$ and $T=\ZZ_p(1)$. We set $r:=[K:\QQ] $. In the following, we set $S:=S_\infty(K)\cup S_p(K)$. We assume that $p$ does not ramify in $K$ and we fix a totally real abelian extension $\cK$ of $K$ that contains $K(\mu_{p^\infty})^+$. We set $G_n:=\Gal(K(\mu_{p^n})^+/K)$ and $\Lambda:=\ZZ_p[[\Gal(K(\mu_{p^\infty})^+/K)]]$. (Note that the notation is slightly different from that in \S \ref{coleman review}.)

We note that Hypothesis \ref{hyp1} is satisfied by the data $ \mathcal{R} = \ZZ_p$ and $T = \ZZ_p(1)$  and with the field $\cK$. Indeed, in this case the stated conditions (iii) and (iv) are obviously satisfied, condition (ii) is satisfied since Kummer theory identifies $H^1(\mathcal{O}_{F,S(F)},\ZZ_p(1))$ with $\mathcal{O}_{F,S(F)}^\times\otimes_\ZZ\ZZ_p$ and this group is torsion-free since $F$ is totally real and $p$ is odd and condition (i) is satisfied since $Y_K(\ZZ_p)$ identifies with the free $\ZZ_p$-module on $S_\infty(K)$ (so that $r(T) = r$ in this case).







For each finite place $v$ of $K$ we write $K(v)$ for the maximal $p$-extension of $K$ inside its ray class field modulo $v$.

For each $F$ in $\Omega(\cK)$ we decompose $\cG_F$ as a direct product $\mathcal{H}_F\times \mathcal{P}_F$ , where $\mathcal{P}_F$ is the Sylow $p$-subgroup of $\cG_F$. For a homomorphism $\chi: \mathcal{H}_F\to \overline \QQ_p^{\times}$ we write $O_\chi$ for the $\ZZ_p$-subalgebra of $\overline \QQ_p$ generated by $\{\chi(h): h \in \mathcal{H}_F\}$ and we use $\chi$ to regard $O_\chi$ as a $\ZZ_p[\mathcal{H}_F]$-module.  For any $\ZZ_p[\cG_F]$-module $M$ we thereby obtain an $O_\chi[\mathcal{P}_F]$-module by setting $M_\chi := O_\chi\otimes_{\ZZ_p[\mathcal{H}_F]}M$.

We write ${\rm Cl}_F$ for the ideal class group of $F$ and also use the  $S(F)$-truncated Selmer module $\mathcal{S}_{S(F)}(\GG_{m/F}) $ of $\GG_{m}$ over $F$, as defined by Kurihara and the present authors in \cite[\S2.1]{bks1}.

For each abelian group $A$ we set $A_p := \ZZ_p\otimes_\ZZ A$.

For any positive even integer $j$, one sees that the $j$-th power of the cyclotomic character $\chi_{\rm cyc}$ induces a map
$$\chi_{\rm cyc}^j: \varprojlim_n {\bigcap}_{\ZZ_p[G_n]}^r H^1(\cO_{K(\mu_{p^n})^+,S}, \ZZ_p(1)) \to {\bigwedge}_{\QQ_p}^r H^1(\cO_{K,S}, \QQ_p(1-j))$$
so that the following diagram commutes:
$$ \small \xymatrix{\displaystyle
 {\rm VS}(\ZZ_p(1) , K(\mu_{p^\infty})^+)^\#=  {\det}_{\Lambda}^{-1}(\rgamma_c(\cO_{K,S}, \Lambda^\#))^\# \ar[r]^{\quad\quad\quad\quad\quad\chi_{\rm cyc}^j} \ar[d]_{\theta_{\ZZ_p(1),K(\mu_{p^\infty})^+}}&\displaystyle {\det}_{\QQ_p}^{-1}(\rgamma_c(\cO_{K,S}, \QQ_p(j))) \ar[d]^{(\ref{cisom}) \text{ and }(\ref{qp})}\\
\displaystyle{\rm ES}_r(\ZZ_p(1), K(\mu_{p^\infty})^+)=\varprojlim_n {\bigcap}_{\ZZ_p[G_n]}^r H^1(\cO_{K(\mu_{p^n})^+,S}, \ZZ_p(1)) \ar[r]_{\quad\quad\quad\quad\quad\quad\quad\chi_{\rm cyc}^j} &\displaystyle{\bigwedge}_{\QQ_p}^r H^1(\cO_{K,S}, \QQ_p(1-j)).
}
$$
Here the upper horizontal arrow is defined in the same way as (\ref{twist map}). (Compare \cite[Cor. 4.9]{tsoi}.) Recall from Definition \ref{gse def} and Remark \ref{even remark} that
$$\eta_{K,S}(j)\in  {\bigwedge}_{\QQ_p}^{r} H^1(\cO_{K,S},\QQ_p(1-j)) $$
denotes the generalized Stark element.

We can now state our main result concerning Euler systems.

\begin{theorem}\label{main}
There exists an Euler system $c=(c_F)_{F \in \Omega(\cK)}$ in ${\rm ES}_r(\ZZ_p(1), \cK)$ that has all of the following properties.
\begin{itemize}
\item[(i)] At every positive even integer $j$ the system $c$ interpolates the special values $\zeta_{K,S}(j)$ in the following sense: the map
$$\chi_{\rm cyc}^j: \varprojlim_n {\bigcap}_{\ZZ_p[G_n]}^r H^1(\cO_{K(\mu_{p^n})^+,S}, \ZZ_p(1)) \to {\bigwedge}_{\QQ_p}^r H^1(\cO_{K,S}, \QQ_p(1-j))$$
sends $(c_{K(\mu_{p^n})^+})_n$ to $\pm \eta_{K,S}(j)$. 
\item[(ii)] For each field $F$ in $\Omega(\cK)$ one has
$$\im (c_F) = {\rm Fitt}_{\ZZ_p[\cG_F]}^0(H^2(\cO_{F,S(F)}, \ZZ_p(1))) = {\rm Fitt}_{\ZZ_p[\cG_F]}^r(\mathcal{S}_{S(F)}(\GG_{m/F})_p^\#).$$
\item[(iii)] Let $F$ be a field in $\Omega(\cK)$ containing the Hilbert $p$-classfield of $K$ and $\chi: \mathcal{H}_F\to \overline \QQ_p^{\times}$ a homomorphism that is not trivial on the decomposition subgroup  of any place in $S(F)_f$. Fix a natural number $n$ and assume that $\cK$ contains $K(v)$ for every place $v$ that splits completely in $F(\mu_{p^n},(\mathcal{O}_K^\times)^{1/p^n})$. Then the $r$-th Kolyvagin derivative of $c$ determines all higher Fitting ideals of the $O_\chi[\mathcal{P}_F]$-module
 $({\rm Cl}_{F}/p^n)_\chi$.
\end{itemize}
\end{theorem}

\subsection{The proof of Theorem \ref{main}}
 At the outset we note that 
there are identifications
\[ \varprojlim_n  {\det}_{\ZZ_p[G_n]}^{-1}(\rgamma(K_p(\mu_{p^n})^+,\ZZ_p ))={\rm VS}^{\rm loc}(\ZZ_p, K(\mu_{p^\infty})^+)\]
(cf. \S \ref{labeling}) and
\[ {\det}_{\Lambda}^{-1}(\rgamma_c(\cO_{K,S}, \Lambda^\#(1) )) ={\rm VS}(\ZZ_p, K(\mu_{p^\infty})^+).\]

We may therefore regard the element $\fz$ constructed in Theorem \ref{IMC} as an element of ${\rm VS}(\ZZ_p,K(\mu_{p^\infty})^+)$ and also use the construction (\ref{phi def}) to define an element
$$\Psi:=\sigma_{D_K} \cdot e^+ \Phi_{\bm{x}} \in {\rm VS}^{\rm loc}(\ZZ_p, K(\mu_{p^\infty})^+)^\#.$$
Here $\sigma_{D_K}$ denotes the unique element of $\Gal(K(\mu_{p^\infty})/K)$ such that $\chi_{\rm cyc}(\sigma_{D_K})=D_K$ (such an element exists since $D_K$ belongs to $\ZZ_p^\times$ as $p$ does not ramify in $K$).

We may then choose a $\ZZ_p[[\Gal(\cK/K)]]$-basis element
$$\widetilde \fz \in {\rm VS}(\ZZ_p,\cK)$$
that the natural (surjective) projection map
$${\rm VS}(\ZZ_p, \cK) \twoheadrightarrow {\rm VS}(\ZZ_p ,K(\mu_{p^\infty})^+)$$
sends to $\fz$. Similarly, we choose a $\ZZ_p[[\Gal(\cK/K)]]$-basis element
$$\widetilde {\rm \Psi} \in {\rm VS}^{\rm loc}(\ZZ_p,\cK)^\#$$
that the natural (surjective) projection map
$${\rm VS}^{\rm loc}(\ZZ_p, \cK)^\# \twoheadrightarrow {\rm VS}^{\rm loc}(\ZZ_p ,K(\mu_{p^\infty})^+)^\#$$
sends to $\Psi$.


Then, by Theorem \ref{funceq}, we can regard $\widetilde \Psi$ as a map (which is an isomorphism, since $\widetilde \Psi $ is a basis) of the form
$$\widetilde \Psi :{\rm VS}(\ZZ_p(1),\cK)^\# \xrightarrow{\sim} {\rm VS}(\ZZ_p,\cK).$$

We may therefore use the homomorphism $\theta_{\ZZ_p(1),\cK}$ from \S\ref{general construct} to define an Euler system
$$c :=\theta_{\ZZ_p(1),\cK} \left( \widetilde \Psi^{-1}(\widetilde \fz)\right)$$
in ${\rm ES}_r(\ZZ_p(1), \cK).$

To verify that this system has the interpolation property in claim (i) we set
$$z:=\Psi^{-1}(\fz) \in {\rm VS}(\ZZ_p(1), K(\mu_{p^\infty})^+)^\# = {\det}^{-1}_\Lambda(\rgamma_c(\cO_{K,S},\Lambda^\#))^\#.$$
Then it is sufficient to prove that for every positive even integer $j$ one has $\vartheta_{\ZZ_p(j)} (\chi_{\rm cyc}^{j}(z)) =\pm \zeta_{K,S}(j)$. ($\vartheta_{\ZZ_p(j)}$ is as in \S \ref{period section}.)

However, this follows directly from the commutativity of diagram (\ref{fe}) by noting that
$$\vartheta_j^{\rm loc}(\chi_{\rm cyc}^j (\Psi))= \pm \frac{\zeta_{K,S}(1-j)}{\zeta_{K,S}(j)} $$
(by Theorem \ref{colint}) and
$$\vartheta_{\ZZ_p(1-j)}(\chi_{\rm cyc}^j(\fz))=\zeta_{K,S}(1-j) $$
(by the interpolation property of the Deligne-Ribet $p$-adic $L$-function $\cL$).

To prove claim (ii) we note that $c$ is, by its very construction, a generator of the $\ZZ_p[[\Gal(\cK/K)]]$-module $\mathcal{E}^{\rm b}(\ZZ_p,\cK)$ of basic Euler systems that is defined in \cite{sbA}. Given this, the first equality in claim (ii) is a direct consequence of \cite[Th. 2.27(ii)]{sbA}.

To prove the second equality in claim (ii) we note first that the definition
\cite[Def. 2.6]{bks1} of the transpose Selmer module $\mathcal{S}^{\rm tr}_{S(F)}(\GG_{m/F})$ combines with the result of \cite[Prop. 2.22(i)]{sbA} (with $\Sigma$ empty) to give an exact sequence of $\ZZ_p[\cG_F]$-modules
\[ 0 \to H^2(\cO_{F,S(F)}, \ZZ_p(1)) \to \mathcal{S}^{\rm tr}_{S(F)}(\GG_{m/F})_p \to \ZZ_p[\cG_F]^r\to 0.\]
By a general property of Fitting ideals, this sequence implies that
\[ {\rm Fitt}_{\ZZ_p[\cG_F]}^0(H^2(\cO_{F,S(F)}, \ZZ_p(1))) = {\rm Fitt}_{\ZZ_p[\cG_F]}^r(\mathcal{S}^{\rm tr}_{S(F)}(\GG_{m/F})_p) = {\rm Fitt}_{\ZZ_p[\cG_F]}^r(\mathcal{S}_{S(F)}(\GG_{m/F})_p^\#)\]
where the second equality is a consequence of \cite[Lem. 2.8]{bks1}. This completes the proof of claim (ii).

Turning to claim (iii) we recall first that, since $p$ is odd, Kummer theory and class field theory combine to give a canonical exact sequence of $\cG_F$-modules of the form
\[ Y_{F,S(F)_f}/p^n \to {\rm Cl}_{F}/p^n \to H^2(\cO_{F,S(F)}, \mu_{p^n}) \to Y_{F,S(F)_f}/p^n.\]
In particular, since for any homomorphism $\chi$ as in (iii) the module $(Y_{F,S(F)_f}/p^n)_\chi$ vanishes, this sequence implies that the $O_\chi[\mathcal{P}_F]$-modules $({\rm Cl}_{F}/p^n)_\chi$ and $H^2(\cO_{F,S(F)}, \mu_{p^n})_\chi$ are isomorphic.

Further, from \cite[Rem. 3.11]{sbA} one knows that the present hypotheses on $\chi$ imply that $H^2(\cO_{F,S(F)}, \mu_{p^n})_\chi$ identifies with the Tate-Shafarevich group $\sha^2(\cO_{K,S(F)},\mathcal{A})$ of the $O_\chi[\mathcal{P}_F]$-module $\mathcal{A} := ({\rm Ind}_{G_F}^{G_K}\mu_{p^n})_\chi$.

To prove claim (iii) it is thus enough to show that the $r$-th Kolyvagin derivative of $c$  determines all Fitting ideals of the  $(\ZZ/p^n)[\mathcal{P}_F]$-module $\sha^2(\cO_{K,S(F)},\mathcal{A})$. But, since $c$ is a generator of $\mathcal{E}^{\rm b}(\ZZ_p,\cK)$, this follows directly from the general results of \cite[Th. 4.15
 and Th. 4.16]{sbA}.

 This completes the proof of Theorem \ref{main}.

\begin{remark} An analysis of the construction of the Euler system $c$ in Theorem \ref{main} shows that it is canonical up to  multiplication by an element of $\ZZ_p[[\Gal(\cK/K)]]^\times$ whose projection to $\ZZ_p[[\Gal(K(\mu_{p^\infty})^+/K)]]^\times$ is equal to $\pm 1$.\end{remark}


\section{A generalized Coleman-Ihara formula}\label{gci section}

In the previous section, we applied Theorem \ref{colint} (the interpolation property of the higher rank Coleman map) for positive even integers to prove Theorem \ref{main}. In this section, we shall give an application of Theorem \ref{colint} for positive odd integers.

Let $K$ be a totally real field in which $p$ is unramified. We set $r:=[K:\QQ]$ and $S:=S_\infty(K) \cup S_p(K)$. We also fix an odd integer $j>1$ and recall the generalized Stark elements
$$\eta_{K,S}(1-j) 
\in \CC_p {\bigwedge}_{\ZZ_p}^r H^1(\cO_{K,S},\ZZ_p(j)) \text{ and }\eta_{K,S}(j)\in \CC_p$$
from Definition \ref{gse def}. 

We use the homomorphism
$${\rm loc}_p:\CC_p {\bigwedge}_{\ZZ_p}^r H^1(\cO_{K,S},\ZZ_p(j)) \to \CC_p {\bigwedge}_{\ZZ_p}^r H^1(K_p,\ZZ_p(j))$$
that is induced by the localization map $K \to K_p$. One easily sees that ${\rm loc}_p$ is non-zero if and only if $H^2(\cO_{K,S},\QQ_p(1-j))$ vanishes (see Remark \ref{explicit stark}(ii)).

In the following result we also use the canonical isomorphism
$${\det}^{-1}_{\QQ_p}(\rgamma(K_p,\QQ_p(j))) \stackrel{(\ref{qp})}{\simeq} {\bigwedge}_{\QQ_p}^r H^1(K_p, \QQ_p(j))$$
to regard the higher rank Coates-Wiles homomorphism
$$\Phi_j \in  {\det}^{-1}_{\ZZ_p}(\rgamma(K_p,\ZZ_p(j)))$$
in Definition \ref{defCW} as an element of ${\bigwedge}_{\QQ_p}^r H^1(K_p, \QQ_p(j))$.

Recall that $D_K$ denotes the discriminant of $K$.

\begin{theorem}\label{CI}
For each odd integer $j>1$ one has
$${\rm loc}_p\left(\eta_{K,S}(1-j)\right)=\pm \eta_{K,S}(j)\cdot D_K^j \cdot \Phi_j \text{ in }\CC_p {\bigwedge}_{\ZZ_p}^r H^1(K_p,\ZZ_p(j)).$$
\end{theorem}

\begin{proof}
We may assume $H^2(\cO_{K,S},\QQ_p(1-j))$ vanishes, since otherwise the claimed equality is trivial.
By Theorem \ref{colint} we have
$$\vartheta_j^{\rm loc}(D_K^j\cdot \Phi_j) = \pm \frac{\zeta_{K,S}(1-j)}{\zeta_{K,S}(j)}.$$
Noting this, the claimed result follows immediately from the explicit definition of the elements $\eta_{K,S}(1-j)$ and $\eta_{K,S}(j)$ and the  commutativity of diagram  (\ref{fe}).
\end{proof}

Upon combining this result with Conjecture \ref{pBC} (which predicts the equality $\eta_{K,S}(j)=\chi_{\rm cyc}^{1-j}(\cL)(=L_p(\omega^{1-j},j)) $), we are led to the following prediction.

\begin{conjecture}[Generalized Coleman-Ihara formula]\label{GCI}
Let $j>1$ be an odd integer. Then we have
$${\rm loc}_p\left(\eta_{K,S}(1-j) \right)=\pm \chi^{1-j}_{\rm cyc}(\cL)\cdot D_K^j \cdot \Phi_j \text{ in }\CC_p {\bigwedge}_{\ZZ_p}^r H^1(K_p,\ZZ_p(j)).$$
\end{conjecture}

From Theorem \ref{CI} and (the final observation of) Remark \ref{abel} one knows that this conjecture is valid whenever $K$ is an abelian extension of $\QQ$.

The interest of the conjecture is explained by the following result which shows that the prediction constitutes a natural generalization of the classical Coleman-Ihara formula (cf. \cite[p. 105]{ihara} or \cite{NSW}) from the case $K = \QQ$ to the case of totally real fields $K$.

\begin{proposition}\label{CIP}
If $K=\QQ$, then the equality of Conjecture \ref{GCI} is equivalent (up to sign) to the Coleman-Ihara formula.
\end{proposition}

\begin{proof} Consider the case $K=\QQ$. Then by Remark \ref{remCW} we have
$$\Phi_j =\pm (p^{j-1}-1)\varphi_j^{\rm CW} \text{ in }H^1(\QQ_p, \QQ_p(j)).$$
Also, we know that
$$\eta_{\QQ,S}(1-j) =- c_j \text{ in }H^1(\ZZ_S, \QQ_p(j)),$$
where $c_j:=c_j(1)$ is the cyclotomic element of Deligne-Soul\'e (see \cite[Def. 3.1.2]{HK}). In fact, by definition, $\eta_{\QQ,S}(1-j)$ is characterized by
$$\lambda_{j}(\eta_{\QQ,S}(1-j)) =\zeta_{\QQ, S}^\ast(1-j)\cdot 2(2\pi \sqrt{-1})^{j-1},$$
where $\lambda_{j} :\CC_p H^1(\ZZ_S, \ZZ_p(j)) \simeq \CC_p K_{2j-1}(\ZZ) \simeq \CC_p H_\QQ(j-1)^+$ denotes the composition ${\rm reg}_j \circ {\rm ch}_j^{-1}$, and we also know that $\lambda_{j}(c_j)=- \zeta_{\QQ, S}^\ast(1-j) \cdot 2 (2\pi\sqrt{-1})^{j-1}$ (see \cite[Th. 5.2.1 and 5.2.2]{HK}). Here we remark that $2 (2\pi\sqrt{-1})^{j-1} =(1+c)(2\pi\sqrt{-1})^{j-1}$ is the canonical basis of $H_\QQ(j-1)^+$ (see \cite[(2)]{bks2-2}).

Thus Conjecture \ref{GCI} in this case is equivalent (up to sign) to the formula
$${\rm loc}_p(c_j) = L_p(\omega^{1-j},j)\cdot (p^{j-1}-1)\varphi_j^{\rm CW} \text{ in }H^1(\QQ_p, \QQ_p(j)),$$
which is exactly the formulation of the Coleman-Ihara formula.
\end{proof}






\begin{acknowledgments}
The authors would like to thank Ryotaro Sakamoto for showing them an earlier version of his article \cite{sakamoto}, Masato Kurihara for many stimulating discussions concerning this, and related, projects, and Rob de Jeu for helpful discussions concerning Schneider's conjecture. The second author would also like to thank Kenji Sakugawa for helpful discussions concerning the Coleman-Ihara formula.
\end{acknowledgments}

\end{document}